\documentclass[11pt]{amsart}

\usepackage{amssymb}
\usepackage{amsmath}
\usepackage{amscd}
\usepackage{float}
\usepackage{graphicx}
\usepackage{amsfonts}
\usepackage{pb-diagram}

\newtheorem{thm}{Theorem}

\newtheorem{lemma}{Lemma}
\newtheorem{prop}{Proposition}
\newtheorem{defn}{Definition}
\newtheorem{remark}{Remark}

\newtheorem{example}{Example}
\newtheorem{nt}{Notation}

\begin{document}

\title[A new basis for the Homflypt skein module of the solid torus]
  {A new basis for the Homflypt skein module of the solid torus}

\author{Ioannis Diamantis}
\address{Department of Mathematics, National Technical University of Athens, Zografou Campus, GR-15780 Athens, Greece.}
\email{diamantis@math.ntua.gr}

\author{Sofia Lambropoulou}
\address{Departament of Mathematics, National Technical University of Athens, Zografou campus, GR-15780 Athens, Greece.}
\email{sofia@math.ntua.gr}
\urladdr{http://www.math.ntua.gr/~sofia}

\keywords{Homflypt skein module, solid torus, Iwahori--Hecke algebra of type B, mixed links, mixed braids, lens spaces. }

\subjclass[2010]{57M27, 57M25, 57N10, 20F36, 20C08}

\thanks{This research  has been co-financed by the European Union (European Social Fund - ESF) and Greek national funds through the Operational Program
``Education and Lifelong Learning" of the National Strategic Reference Framework (NSRF) - Research Funding Program: THALES: Reinforcement of the
interdisciplinary and/or inter-institutional research and innovation. }

\setcounter{section}{-1}

\date{}

\begin{abstract}
In this paper we give a new basis, $\Lambda$, for the Homflypt skein module of the solid torus, $\mathcal{S}({\rm ST})$, which was  predicted by Jozef Przytycki using topological interpretation. The basis $\Lambda$ is different from the basis $\Lambda^{\prime}$, discovered independently by Hoste--Kidwell \cite{HK} and Turaev \cite{Tu} with the use of diagrammatic methods, and also different from the basis of Morton--Aiston \cite{MA}. For finding the basis $\Lambda$ we use the generalized Hecke algebra of type B, $\textrm{H}_{1,n}$, defined by the second author in \cite{La2}, which is generated by looping elements and braiding elements and which is isomorphic to the affine Hecke algebra of type A. Namely, we start with the well-known basis of $\mathcal{S}({\rm ST})$, $\Lambda^{\prime}$, and an appropriate linear basis $\Sigma_n$ of the algebra $\textrm{H}_{1,n}$. We then convert elements in $\Lambda^{\prime}$ to linear combinations of elements in the new basic set $\Lambda$. This is done in two steps: First we convert elements in $\Lambda^{\prime}$ to elements in $\Sigma_n$. Then, using conjugation and the stabilization moves, we convert these elements to linear combinations of elements in $\Lambda$ by managing gaps in the indices of the looping elements and by eliminating braiding tails in the words. Further, we define an ordering relation in $\Lambda^{\prime}$ and $\Lambda$ and prove that the sets are totally ordered. Finally, using this ordering, we relate the sets $\Lambda^{\prime}$ and $\Lambda$ via a block diagonal matrix, where each block is an infinite lower triangular matrix with invertible elements in the diagonal and we prove linear independence of the set $\Lambda$. The infinite matrix is then ``invertible'' and thus, the set $\Lambda$ is a basis for $\mathcal{S}({\rm ST})$.

$\mathcal{S}({\rm ST})$ plays an important role in the study of Homflypt skein modules of arbitrary c.c.o. $3$-manifolds, since every c.c.o. $3$-manifold can be obtained by integral surgery along a framed link in $S^3$ with unknotted components. The new basis, $\Lambda$, of $\mathcal{S}({\rm ST})$ is appropriate for computing the Homflypt skein module of the lens spaces. The aim of this paper is to provide the basic algebraic tools for computing skein modules of c.c.o. 3-manifolds via algebraic means.
\end{abstract}

\maketitle

\section{Introduction}\label{intro}

Let $M$ be an oriented $3$-manifold, $R=\mathbb{Z}[u^{\pm1},z^{\pm1}]$, $\mathcal{L}$ the set of all oriented links in $M$ up to ambient isotopy in $M$ and
let $S$ the submodule of $R\mathcal{L}$ generated by the skein expressions $u^{-1}L_{+}-uL_{-}-zL_{0}$, where $L_{+}$, $L_{-}$ and $L_{0}$ are oriented links that have identical diagrams, except in one crossing, where they are as depicted in Figure~\ref{skein}.

\begin{figure}[!ht]
\begin{center}
\includegraphics[width=1.7in]{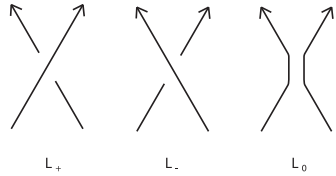}
\end{center}
\caption{The links $L_{+}, L_{-}, L_{0}$ locally.}
\label{skein}
\end{figure}

\noindent For convenience we allow the empty knot, $\emptyset$, and add the relation $u^{-1} \emptyset -u\emptyset =zT_{1}$, where $T_{1}$ denotes the trivial
knot. Then the {\it Homflypt skein module} of $M$ is defined to be:

\begin{equation*}
\mathcal{S} \left(M\right)=\mathcal{S} \left(M;{\mathbb Z}\left[u^{\pm 1} ,z^{\pm 1} \right],u^{-1} L_{+} -uL_{-} -zL{}_{0} \right)={\raise0.7ex\hbox{$
R\mathcal{L} $}\!\mathord{\left/ {\vphantom {R\mathcal{L} S }} \right. \kern-\nulldelimiterspace}\!\lower0.7ex\hbox{$ S  $}}.
\end{equation*}

\smallbreak

Unlike the Kauffman bracket skein module, the Homflypt skein module of a $3$-manifold, also known as \textit{Conway skein module} and as \textit{third skein module}, is very hard to compute (see  [P-2] for the case of the product of a surface and the interval).

\smallbreak

Let ST denote the solid torus. In \cite{Tu}, \cite{HK} the Homflypt skein module of the solid torus has been computed using diagrammatic methods by means of the following theorem:

\begin{thm}[Turaev, Kidwell--Hoste] \label{turaev}
The skein module $\mathcal{S}({\rm ST})$ is a free, infinitely generated $\mathbb{Z}[u^{\pm1},z^{\pm1}]$-module isomorphic to the symmetric
tensor algebra $SR\widehat{\pi}^0$, where $\widehat{\pi}^0$ denotes the conjugacy classes of non trivial elements of $\pi_1(\rm ST)$.
\end{thm}

A basic element of $\mathcal{S}({\rm ST})$ in the context of \cite{Tu, HK}, is illustrated in Figure~\ref{tur}. In the diagrammatic setting of \cite{Tu} and \cite{HK}, ST is considered as ${\rm Annulus} \times {\rm Interval}$. The Homflypt skein module of ST is particularly important, because any closed, connected, oriented (c.c.o.) $3$-manifold can be obtained by surgery along a framed link in $S^3$ with unknotted components.

\begin{figure}
\begin{center}
\includegraphics[width=1.3in]{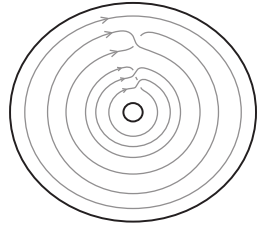}
\end{center}
\caption{A basic element of $\mathcal{S}({\rm ST})$.}
\label{tur}
\end{figure}

\smallbreak

A different basis of $\mathcal{S}({\rm ST})$, known as Young idempotent basis, is based on the work of Morton and Aiston \cite{MA} and Blanchet \cite{B}.

\smallbreak

In \cite{La2}, $\mathcal{S}({\rm ST})$ has been recovered using algebraic means. More precisely, the generalized Hecke algebra of type B, $\textrm{H}_{1,n}(q)$, is introduced, which is isomorphic to the affine Hecke algebra of type A, $\widetilde{\textrm{H}_n}(q)$. Then, a unique Markov trace is constructed on the algebras $\textrm{H}_{1,n}(q)$ leading to an invariant for links in ST, the universal analogue of the Homflypt polynomial for ST. This trace gives distinct values on distinct elements of the \cite{Tu, HK}-basis of $\mathcal{S}({\rm ST})$. The link isotopy in ST, which is taken into account in the definition of the skein module and which corresponds to conjugation and the stabilization moves on the braid level, is captured by the the conjugation property and the Markov property of the trace, while the defining relation of the skein module is reflected into the quadratic relation of $\textrm{H}_{1,n}(q)$. In the algebraic language of \cite{La2} the basis of $\mathcal{S}({\rm ST})$, described in Theorem~\ref{turaev}, is given in open braid form by the set $\Lambda^{\prime}$ in Eq.~\ref{Lpr}. Figure~\ref{els3} illustrates the basic element of Figure~\ref{tur} in braid notation. Note that in the setting of \cite{La2} ST is considered as the complement of the unknot (the bold curve in the figure). The looping elements $t^{\prime}_i \in \textrm{H}_{1,n}(q)$ in the monomials of $\Lambda^{\prime}$ are all conjugates, so they are consistent with the trace property and they enable the definition of the trace via simple inductive rules.

\smallbreak

In this paper we give a new basis $\Lambda$ for $\mathcal{S}({\rm ST})$ conjectured by the J.~H.~Przytycki, using the algebraic methods developed in \cite{La2}. The motivation of this work is the computation of $\mathcal{S} \left( L(p,q) \right)$ via algebraic means. The new basic set is described in Eq.~\ref{basis} in open braid form. The looping elements $t_i$ are in the algebras $\textrm{H}_{1,n}(q)$ and they are commuting. For a comparative illustration and for the defining formulas of the $t_i$'s and the $t_i^{\prime}$'s the reader is referred to Figure~\ref{genh} and Eq.~\ref{lgen} respectively. Moreover, the $t_i$'s are consistent with the handle sliding move or band move used in the link isotopy in $L(p,q)$, in the sense that a braid band move can be described naturally with the use of the $t_i$'s (see for example \cite{DL} and references therein).

Our main result is the following:

\begin{thm}\label{mainthm}
The following set is a $\mathbb{Z}[q^{\pm1}, z^{\pm1}]$-basis for $\mathcal{S}({\rm ST})$:
\begin{equation}\label{basis}
\Lambda=\{t^{k_0}t_1^{k_1}\ldots t_n^{k_n},\ k_i\in \mathbb{Z}\setminus\{0\}\ \forall i,\ n \in \mathbb{N} \}.
\end{equation}
\end{thm}

\begin{figure}
\begin{center}
\includegraphics[width=1.3in]{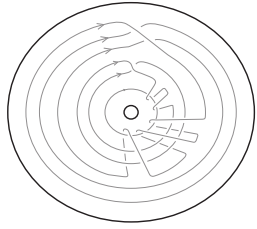}
\end{center}
\caption{An element of the new basis $\Lambda$.}
\label{tbasis}
\end{figure}

Our method for proving Theorem~\ref{mainthm} is the following:

\smallbreak

$\bullet$ We define total orderings in the sets $\Lambda^{\prime}$ and $\Lambda$ and

\smallbreak

$\bullet$ we show that the two ordered sets are related via a lower triangular infinite matrix with invertible elements on the diagonal.

\smallbreak

More precisely, two analogous sets, $\Sigma_n$ and $\Sigma^{\prime}_n$, are given in \cite{La2} as linear bases for the algebra $\textrm{H}_{1,n}(q)$. See Theorem~\ref{basesH} in this paper. The set $\bigcup_n \Sigma_n$ includes $\Lambda$ as a proper subset and the set $\bigcup_n \Sigma^{\prime}_n$ includes $\Lambda^{\prime}$ as a proper subset. The sets $\Sigma_n$ come directly from the works of S.~Ariki and K.~Koike, and M.~Brou\`{e} and G.~Malle on the cyclotomic Hecke algebras of type B. See \cite{La2} and references therein. The second set $\bigcup_n \Sigma^{\prime}_n$ includes $\Lambda^{\prime}$ as a proper subset. The sets $\Sigma^{\prime}_n$ appear naturally in the structure of the braid groups of type B, $B_{1,n}$; however, it is very complicated to show that they are indeed basic sets for the algebras $\textrm{H}_{1,n}(q)$. The sets $\Sigma_n$ play an intrinsic role  in the proof of Theorem~\ref{mainthm}. Indeed, when trying to convert a monomial $\lambda^{\prime}$ from $\Lambda^{\prime}$ into a linear combination of elements in $\Lambda$ we pass by elements of the sets $\Sigma_n$. This means that in the converted expression of $\lambda^{\prime}$ we have monomials in the $t_i$'s, with possible gaps in the indices followed by monomials in the braiding generators $g_i$. So, in order to reach expressions in the set $\Lambda$ we need:

\smallbreak

$\bullet$ to manage the gaps in the indices of the $t_i$'s and

\smallbreak

$\bullet$ to eliminate the braiding `tails'.

\smallbreak

The paper is organized as follows. In Section~1 we recall the algebraic setting and the results needed from \cite{La2}. In Section~2 we define the orderings in the two sets $\Lambda$ and $\Lambda^{\prime}$ and we prove that the sets are totally ordered. In Section~3 we prove a series of lemmas for converting elements in $\Lambda^{\prime}$ to elements in the sets $\Sigma_n$. In Section~4 we convert elements in $\Sigma_n$ to elements in $\Lambda$ using conjugation and the stabilization moves. Finally in Section~5 we prove that the sets $\Lambda^{\prime}$ and $\Lambda$ are related through a lower triangular infinite matrix mentioned above. A computer program converting elements in $\Lambda^{\prime}$ to elements in $\Sigma_n$ has been developed by K. Karvounis and will be soon available on $http://www.math.ntua.gr/^{\sim}sofia$.

\smallbreak

The algebraic techniques developed here will serve as basis for computing Homflypt skein modules of arbitrary c.c.o. $3$-manifolds using the braid approach. The advantage of this approach is that we have an already developed homogeneous theory of braid structures and braid equivalences for links in c.c.o. $3$-manifolds (\cite{LR1, LR2, DL}). In fact, these algebraic techniques are used and developed further in \cite{KL} for knots and links in $3$-manifolds represented by the $2$-unlink.

\section{The Algebraic Settings}

\subsection{Mixed Links in $S^3$}

We now view ST as the complement of a solid torus in $S^3$. An oriented link $L$ in ST can be represented by an oriented \textit{mixed link} in $S^{3}$, that is a link in $S^{3}$ consisting of the unknotted fixed part $\widehat{I}$ representing the complementary solid torus in $S^3$ and the moving part $L$ that links
with $\widehat{I}$.

\smallbreak

A \textit{mixed link diagram }is a diagram $\widehat{I}\cup \widetilde{L}$ of $\widehat{I}\cup L$ on the plane of $\widehat{I}$, where this plane is equipped with the top-to-bottom direction of $I$.

\begin{figure}
\begin{center}
\includegraphics[width=1.1in]{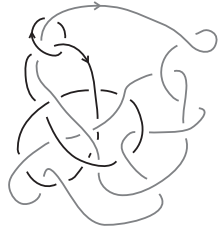}
\end{center}
\caption{A mixed link in $S^3$.}
\label{mlink}
\end{figure}

\smallbreak

Consider now an isotopy of an oriented link $L$ in ST. As the link moves in ST, its corresponding mixed link will change in $S^{3}$ by a sequence
of moves that keep the oriented $\widehat{I}$ pointwise fixed. This sequence of moves consists in isotopy in the $S^{3}$ and the \textit{mixed Reidemeister
moves}. In terms of diagrams we have the following result for isotopy in ST:

\smallbreak

The mixed link equivalence in $S^{3}$ includes the classical Reidemeister moves and the mixed Reidemeister moves, which involve the fixed and the
standard part of the mixed link, keeping $\widehat{I}$ pointwise fixed.

\subsection{Mixed Braids in $S^3$}

\noindent By the Alexander theorem for knots in solid torus, a mixed link diagram $\widehat{I}\cup \widetilde{L}$ of $\widehat{I}\cup L$ may be turned into a
\textit{mixed braid} $I\cup \beta $ with isotopic closure. This is a braid in $S^{3}$ where, without loss of generality, its first strand represents
$\widehat{I}$, the fixed part, and the other strands, $\beta$, represent the moving part $L$. The subbraid $\beta$ shall be called the \textit{moving part} of
$I\cup \beta $.

\begin{figure}
\begin{center}
\includegraphics[width=2in]{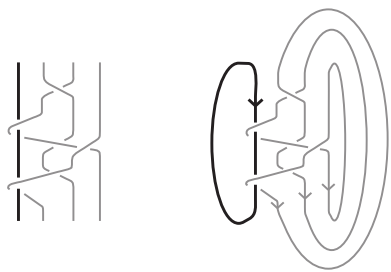}
\end{center}
\caption{The closure of a mixed braid to a mixed link.}
\label{mbtoml}
\end{figure}

The sets of braids related to the ST form groups, which are in fact the Artin braid groups type B, denoted $B_{1,n}$, with presentation:

\[ B_{1,n} = \left< \begin{array}{ll}  \begin{array}{l} t, \sigma_{1}, \ldots ,\sigma_{n-1}  \\ \end{array} & \left| \begin{array}{l}
\sigma_{1}t\sigma_{1}t=t\sigma_{1}t\sigma_{1} \ \   \\
 t\sigma_{i}=\sigma_{i}t, \quad{i>1}  \\
{\sigma_i}\sigma_{i+1}{\sigma_i}=\sigma_{i+1}{\sigma_i}\sigma_{i+1}, \quad{ 1 \leq i \leq n-2}   \\
 {\sigma_i}{\sigma_j}={\sigma_j}{\sigma_i}, \quad{|i-j|>1}  \\
\end{array} \right.  \end{array} \right>, \]

\noindent where the generators $\sigma _{i}$ and $t$ are illustrated in Figure~\ref{gen}.

\begin{figure}
\begin{center}
\includegraphics[width=2.3in]{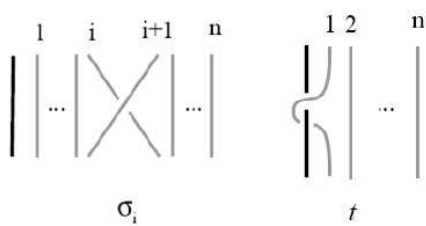}
\end{center}
\caption{The generators of $B_{1,n}$.}
\label{gen}
\end{figure}

Isotopy in ST is translated on the level of mixed braids by means of the following theorem.

\begin{thm}[Theorem~3, \cite{La1}] \label{markov}
 Let $L_{1} ,L_{2}$ be two oriented links in ST and let $I\cup \beta_{1} ,{\rm \; }I\cup \beta_{2}$ be two corresponding mixed braids in $S^{3}$. Then $L_{1}$ is isotopic to $L_{2}$ in ST if and only if $I\cup \beta_{1}$ is equivalent to $I\cup \beta_{2}$ in $\mathop{\cup }\limits_{n=1}^{\infty } B_{1,n}$ by the following moves:

\[ \begin{array}{clll}
(i)  & Conjugation:         & \alpha \sim \beta^{-1} \alpha \beta, & {\rm if}\ \alpha ,\beta \in B_{1,n}. \\
(ii) & Stabilization\ moves: &  \alpha \sim \alpha \sigma_{n}^{\pm 1} \in B_{1,n+1}, & {\rm if}\ \alpha \in B_{1,n}. \\
\end{array} \]
\end{thm}

\subsection{The Generalized Iwahori-Hecke Algebra of type B}

It is well known that $B_{1,n}$ is the Artin group of the Coxeter group of type B, which is related to the Hecke algebra of type B, $\textrm{H}_{n}{(q,Q)}$ and to the cyclotomic Hecke algebras of type B. In \cite{La2} it has been established that all these algebras form a tower of B-type algebras and are related to the knot theory of ST. The basic one is $\textrm{H}_{n}{(q,Q)}$, a presentation of which is obtained from the presentation of the Artin group $B_{1,n}$ by adding the quadratic relations

\begin{equation}\label{quad}
{g_{i}^2=(q-1)g_{i}+q}
\end{equation}

\noindent and the relation $t^{2} =\left(Q-1\right)t+Q$, where $q,Q \in {\mathbb C}\backslash \{0\}$ are seen as fixed variables. The middle B--type algebras are the cyclotomic Hecke algebras of type B, $\textrm{H}_{n}(q,d)$, whose presentations are obtained by the quadratic relation~(\ref{quad}) and $t^d=(t-u_{1})(t-u_{2}) \ldots (t-u_{d})$. The topmost Hecke-like algebra in the tower is the \textit{generalized Iwahori--Hecke algebra of type B}, $\textrm{H}_{1,n}(q)$, which, as observed by T.tom Dieck, is isomorphic to the affine Hecke algebra of type A, $\widetilde{\textrm{H}}_n(q)$ (cf. \cite{La2}). The algebra $\textrm{H}_{1,n}(q)$ has the following presentation:

\[
\textrm{H}_{1,n}{(q)} = \left< \begin{array}{ll}  \begin{array}{l} t, g_{1}, \ldots ,g_{n-1}  \\ \end{array} & \left| \begin{array}{l} g_{1}tg_{1}t=tg_{1}tg_{1} \ \
\\
 tg_{i}=g_{i}t, \quad{i>1}  \\
{g_i}g_{i+1}{g_i}=g_{i+1}{g_i}g_{i+1}, \quad{1 \leq i \leq n-2}   \\
 {g_i}{g_j}={g_j}{g_i}, \quad{|i-j|>1}  \\
 {g_i}^2=(q-1)g_{i}+q, \quad{i=1,\ldots,n-1}
\end{array} \right.  \end{array} \right>.
\]

\noindent That is:

\begin{equation*}
\textrm{H}_{1,n}(q)= \frac{{\mathbb Z}\left[q^{\pm 1} \right]B_{1,n}}{ \langle \sigma_i^2 -\left(q-1\right)\sigma_i-q \rangle}.
\end{equation*}

Note that in $\textrm{H}_{1,n}(q)$ the generator $t$ satisfies no polynomial relation, making the algebra $\textrm{H}_{1,n}(q)$ infinite dimensional. Also that in \cite{La2} the algebra $\textrm{H}_{1,n}(q)$ is denoted as $\textrm{H}_{n}(q, \infty)$.

\smallbreak

In \cite{Jo} V.F.R. Jones gives the following linear basis for the Iwahori-Hecke algebra of type A, $\textrm{H}_{n}(q)$:

{\small

$$
S =\left\{(g_{i_{1} }g_{i_{1}-1}\ldots g_{i_{1}-k_{1}})(g_{i_{2} }g_{i_{2}-1 }\ldots g_{i_{2}-k_{2}})\ldots (g_{i_{p} }g_{i_{p}-1 }\ldots g_{i_{p}-k_{p}})\right\}, \mbox{ for } 1\le i_{1}<\ldots <i_{p} \le n-1{\rm \; }.
$$
}

\noindent The basis $S$ yields directly an inductive basis for $\textrm{H}_{n}(q)$, which is used in the construction of the Ocneanu trace, leading to the Homflypt or $2$-variable Jones polynomial.

In $\textrm{H}_{1,n}(q)$ we define the elements:
\begin{equation}\label{lgen}
t_{i}:=g_{i}g_{i-1}\ldots g_{1}tg_{1} \ldots g_{i-1}g_{i}\ \rm{and}\ t^{\prime}_{i}:=g_{i}g_{i-1}\ldots g_{1}tg_{1}^{-1}\ldots g_{i-1}^{-1}g_{i}^{-1},
\end{equation}
as illustrated in Figure~\ref{genh}.

\smallbreak

In \cite{La2} the following result has been proved.

\begin{thm}[Proposition~1, Theorem~1 \cite{La2}] \label{basesH}
The following sets form linear bases for ${\rm H}_{1,n}(q)$:
\[
\begin{array}{llll}
 (i) & \Sigma_{n} & = & t_{i_{1} } ^{k_{1} } t_{i_{2} } ^{k_{2} } \ldots t_{i_{r}}^{k_{r} } \cdot \sigma ,\ {\rm where}\ 1\le i_{1} <\ldots <i_{r} \le n-1,\\
     &            &   &                                                                                                                                       \\
 (ii) & \Sigma^{\prime} _{n} & = &  {t^{\prime}_{i_1}}^{k_{1}} {t^{\prime}_{i_2}}^{k_{2}} \ldots {t^{\prime}_{i_r}}^{k_{r}} \cdot \sigma ,\ {\rm where}\ 1\le i_{1} < \ldots <i_{r} \le n, \\
\end{array}
\]
\noindent where $k_{1}, \ldots ,k_{r} \in {\mathbb Z}$ and $\sigma$ a basic element in $\textrm{H}_{n}(q)$.
\end{thm}

\begin{figure}
\begin{center}
\includegraphics[width=3.2in]{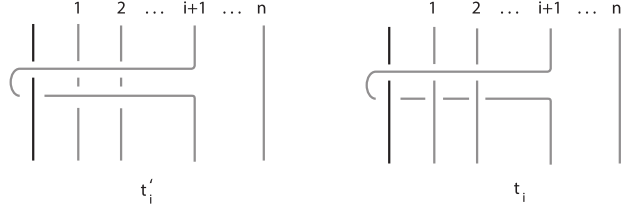}
\end{center}
\caption{The elements $t^{\prime}_{i}$ and $t_{i}$.}
\label{genh}
\end{figure}

\begin{remark}\label{conind}\rm
\begin{itemize}
\item[(i)] The indices of the $t^{\prime}_i$'s in the set $\Sigma^{\prime}_n$ are ordered but are not necessarily consecutive, neither do they
need to start from $t$.
\item[(ii)] A more straight forward proof that the sets $\Sigma_n^{\prime}$ form bases for $\textrm{H}_{1,n}(q)$ can be found in \cite{D}.
\end{itemize}
\end{remark}

In \cite{La2} the basis $\Sigma^{\prime}_{n}$ is used for constructing a Markov trace on $\bigcup _{n=1}^{\infty }\textrm{H}_{1,n}(q)$.

\begin{thm}[Theorem~6, \cite{La2}] \label{tr}
Given $z,s_{k}$, with $k\in {\mathbb Z}$ specified elements in $R={\mathbb Z}\left[q^{\pm 1} \right]$, there exists
a unique linear Markov trace function
\begin{equation*}
{\rm tr}:\bigcup _{n=1}^{\infty }{\rm H}_{1,n}(q)  \to R\left(z,s_{k} \right),k\in {\mathbb Z}
\end{equation*}

\noindent determined by the rules:

\[
\begin{array}{lllll}
(1) & {\rm tr}(ab) & = & {\rm tr}(ba) & \quad {\rm for}\ a,b \in {\rm H}_{1,n}(q) \\
(2) & {\rm tr}(1) & = & 1 & \quad \textrm{for}\ all\ {\rm H}_{1,n}(q) \\
(3) & {\rm tr}(ag_{n}) & = & z{\rm tr}(a) & \quad \textrm{for}\ a \in {\rm H}_{1,n}(q) \\
(4) & {\rm tr}(a{t^{\prime}_{n}}^{k}) & = & s_{k}{\rm tr}(a) & \quad \textrm{for}\ a \in {\rm H}_{1,n}(q),\ k \in {\mathbb Z}. \\
\end{array}
\]
\end{thm}

Note that the use of the looping elements $t_i^{\prime}$ enable the trace ${\rm tr}$ to be defined by just extending the three rules of the Ocneanu trace on the algebras ${\rm H}_n(q)$ \cite{Jo} by rule (4). Using $\textrm{tr}$ Lambropoulou constructed a universal Homflypt-type invariant for oriented links in ST. Namely, let $\mathcal{L}$ denote the set of oriented links in ST. Then:

\begin{thm} [Definition~1, \cite{La2}] \label{inv}
The function $X:\mathcal{L}$ $\rightarrow R(z,s_{k})$

\begin{equation*}
X_{\widehat{\alpha}}=\left[-\frac{1-\lambda q}{\sqrt{\lambda } \left(1-q\right)} \right]^{n-1} \left(\sqrt{\lambda } \right)^{e}
{\rm tr}\left(\pi \left(\alpha \right)\right),
\end{equation*}

\noindent where $\alpha \in B_{1,n}$ is a word in the $\sigma _{i}$'s and $t^{\prime}_{i} $'s, $e$ is the exponent sum of the $\sigma _{i}$'s in $\alpha $, and
$\pi$ the canonical map of $B_{1,n}$ in ${\rm H}_{1,n}(q)$, such that $t\mapsto t$ and $\sigma _{i} \mapsto g_{i} $, is an invariant of oriented links in ST.
\end{thm}

The invariant $X$ satisfies a skein relation \cite{La1}. Theorems~\ref{basesH}, \ref{tr} and \ref{inv} hold also for the algebras $\textrm{H}_{n}(q,Q)$ and $\textrm{H}_{n}(q,d)$, giving rise to all possible Homflypt-type invariants for knots in ST. For the case of the Hecke algebra of type B, $\textrm{H}_{n}(q,Q)$, see also \cite{La1} and \cite{LG}.

\subsection{The basis of $\mathcal{S}({\rm ST})$ in algebraic terms}

Let us now see how $\mathcal{S}({\rm ST})$ is described in the above algebraic language. We note first that an element $\alpha$ in the basis of $\mathcal{S}({\rm ST})$ described in Theorem~\ref{turaev} when ST is considered as ${\rm Annulus} \times {\rm Interval}$, can be illustrated equivalently as a mixed link in $S^3$ when ST is viewed as the complement of a solid torus in $S^3$. So we correspond the element $\alpha$ to the minimal mixed braid representation, which has increasing order of twists around the fixed strand. Figure~\ref{els3} illustrates an example of this correspondence. Denoting

\begin{equation}\label{Lpr}
\Lambda^{\prime}=\{ {t^{\prime}_1}^{k_1}{t^{\prime}_2}^{k_2} \ldots
{t^{\prime}_n}^{k_n}, \ k_i \in \mathbb{Z}\setminus\{0\} \ \forall i,\ n\in \mathbb{N} \},
\end{equation}

\noindent we have that $\Lambda^{\prime}$ is a subset of $\bigcup_n{\textrm{H}_{1,n}}$. In particular $\Lambda^{\prime}$ is a subset of $\bigcup_n{\Sigma^{\prime}_n}$.

\begin{figure}
\begin{center}
\includegraphics[width=3.3in]{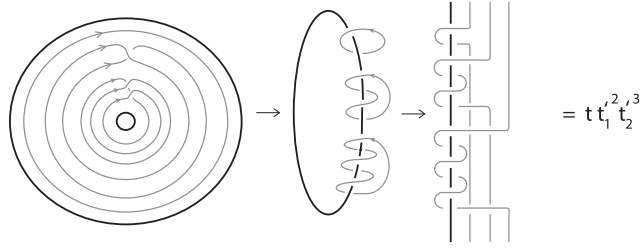}
\end{center}
\caption{An element in $\Lambda^{\prime}$.}
\label{els3}
\end{figure}

\begin{figure}
\begin{center}
\includegraphics[width=3.5in]{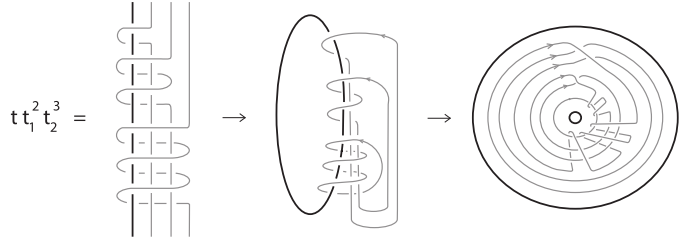}
\end{center}
\caption{An element of $\Lambda$.}
\label{els4}
\end{figure}

Applying the inductive trace rules to a word $w$ in $\bigcup_n\Sigma^{\prime}_n$ will eventually give rise to linear combinations of monomials in
$R(z, s_k)$. In particular, for an element of $\Lambda^{\prime}$ we have:

$${\rm tr}(t^{k_0}{t^{\prime}_1}^{k_1}\ldots {t^{\prime}_{n-1}}^{k_{n-1}})=s_{k_{n-1}}\ldots s_{k_{1}}s_{k_{0}}.$$

Further, the elements of $\Lambda^{\prime}$ are in bijective correspondence with increasing $n$-tuples of integers, $(k_0,k_1,\ldots,k_{n-1})$, $n \in
\mathbb{N}$, and these are in bijective correspondence with monomials in $s_{k_0},s_{k_1}, \ldots, s_{k_{n-1}}$.

\begin{remark}
\rm The invariant $X$ recovers the Homflypt skein module of ST since it gives different values for different elements of $\Lambda^{\prime}$ by
rule~4 of the trace.
\end{remark}

\section{An ordering in the sets $\Lambda$ and $\Lambda^{\prime}$}

In this section we define an ordering relation in the sets $\Lambda$ and $\Lambda^{\prime}$. Before that, we will need the notion of the index of a word in $\Lambda^{\prime}$ or in $\Lambda$ .

\begin{defn} \label{index} \rm
The {\it index} of a word $w$ in $\Lambda^{\prime}$ or in $\Lambda$, denoted $ind(w)$, is defined to be the highest index of the $t_i^{\prime}$'s, resp. of the $t_i$'s, in $w$. Similarly, the \textit{index} of an element in $\Sigma_n^{\prime}$ or in $\Sigma_n$ is defined in the same way by ignoring possible gaps in the indices of the looping generators and by ignoring the braiding part in $\textrm{H}_{n}(q)$. Moreover, the index of a monomial in $\textrm{H}_{n}(q)$ is equal to $0$.
\end{defn}

\noindent For example, $ind({t^{\prime}}^{k_0}{t^{\prime}_1}^{k_1}\ldots {t^{\prime}_n}^{k_n})=ind(t^{u_0}\ldots t_n^{u_n})=n$.

\begin{defn} \label{order}
\rm
 We define the following {\it ordering} in the set $\Lambda^{\prime}$. Let $w={t^{\prime}_{i_1}}^{k_1}{t^{\prime}_{i_2}}^{k_2}\ldots {t^{\prime}_{i_{\mu}}}^{k_{\mu}}$ and $\sigma={t^{\prime}_{j_1}}^{\lambda_1}{t^{\prime}_{j_2}}^{\lambda_2}\ldots {t^{\prime}_{j_{\nu}}}^{\lambda_{\nu}}$, where $k_t , \lambda_s \in \mathbb{Z}$, for all $t,s$. Then:

\smallbreak

\begin{itemize}
\item[(a)] If $\sum_{i=0}^{\mu}k_i < \sum_{i=0}^{\nu}\lambda_i$, then $w<\sigma$.

\vspace{.1in}

\item[(b)] If $\sum_{i=0}^{\mu}k_i = \sum_{i=0}^{\nu}\lambda_i$, then:

\vspace{.1in}

\noindent  (i) if $ind(w)<ind(\sigma)$, then $w<\sigma$,

\vspace{.1in}

\noindent  (ii) if $ind(w)=ind(\sigma)$, then:

\vspace{.1in}

\noindent \ \ \ \ ($\alpha$) if $i_1=j_1, i_2=j_2, \ldots , i_{s-1}=j_{s-1}, i_{s}<j_{s}$, then $w>\sigma$,

\vspace{.1in}

\noindent \ \ \ \ ($\beta$) if $i_t=j_t\ \forall t$ and $k_{\mu}=\lambda_{\mu}, k_{\mu-1}=\lambda_{\mu-1}, \ldots k_{\i+1}=\lambda_{i+1}, |k_i|<|\lambda_i|$, then $w<\sigma$,

\vspace{.1in}

\noindent \ \ \ \ ($\gamma$) if $i_t=j_t\ \forall t$ and $k_{\mu}=\lambda_{\mu}, k_{\mu-1}=\lambda_{\mu-1}, \ldots k_{\i+1}=\lambda_{i+1}, |k_i|=|\lambda_i|$ and $k_i>\lambda_i,$

\vspace{.1in}

\noindent \ \ \ \ \ \ \ \ \ then $w<\sigma$,

\vspace{.1in}

\noindent \ \ \ \ ($\delta$) if $i_t=j_t\ \forall t$ and $k_i=\lambda_i$, $\forall i$, then $w=\sigma$.

\vspace{.1in}

\item[(c)] In the general case where $w={t^{\prime}_{i_1}}^{k_1}{t^{\prime}_{i_2}}^{k_2}\ldots {t^{\prime}_{i_{\mu}}}^{k_{\mu}} \cdot \sigma_1$ and $\sigma={t^{\prime}_{j_1}}^{\lambda_1}{t^{\prime}_{j_2}}^{\lambda_2}\ldots {t^{\prime}_{j_{\nu}}}^{\lambda_{\nu}}\cdot \sigma_2$, where $\sigma_1, \sigma_2 \in \textrm{H}_n(q)$, the ordering is defined in the same way by ignoring the braiding parts $\sigma_1, \sigma_2$.
\end{itemize}
\end{defn}

The same ordering is defined on the set $\Lambda$, where the $t^{\prime}_i$'s are replaced by the corresponding $t_i$'s. Moreover, the same ordering is defined on the sets $\Sigma_n$ and $\Sigma_n^{\prime}$ by ignoring the braiding parts.

\begin{prop}
The set $\Lambda^{\prime}$ equipped with the ordering given in Definition~\ref{order}, is totally ordered set.
\end{prop}

\begin{proof}
In order to show that the set  $\Lambda^{\prime}$ is totally ordered set when equipped with the ordering given in Definition~\ref{order}, we need to show that the ordering relation is antisymmetric, transitive and total. We only show that the ordering relation is transitive. Antisymmetric property follows similarly. Totality follows from Definition~\ref{order} since all possible cases have been considered.

\smallbreak

Let $w={t}^{k_0}{t^{\prime}_{1}}^{k_1}\ldots {t^{\prime}_{m}}^{k_m}$, $\sigma={t}^{\lambda_0}{t^{\prime}_{1}}^{\lambda_1} \ldots {t^{\prime}_{n}}^{\lambda_n}$ and $v=t^{\mu_0} {t^{\prime}_{1}}^{\mu_1} \ldots {t^{\prime}_{p}}^{\mu_p}$ and let $w<\sigma$ and $\sigma<v$.

\smallbreak

Since $w<\sigma$, one of the following holds:

\smallbreak

\begin{itemize}
\item[(a)] Either $\sum_{i=1}^{m}k_i<\sum_{i=1}^{n}\lambda_i$ and since $\sigma<v$, we have that $\sum_{i=1}^{n}\lambda_i\leq
\sum_{i=1}^{p}\mu_i$ and so\\

\smallbreak

\noindent $\sum_{i=1}^{m}k_i<\sum_{i=1}^{p}\mu_i$. Thus $w<v$.\\

\bigbreak

\item[(b)] Either $\sum_{i=1}^{m}k_i=\sum_{i=1}^{n}\lambda_i$ and $ind(w)=m<n=ind(\sigma)$. Then, since $\sigma<v$ we have\\

\smallbreak

\noindent that either $\sum_{i=1}^{n}\lambda_i<\sum_{i=1}^{p}\mu_i$ $\big(\textrm{same as in case (a)}\big)$ or $\sum_{i=1}^{n}\lambda_i=\sum_{i=1}^{p}\mu_i$ and \\

\smallbreak

\noindent $ind(\sigma)\leq p = ind(v)$. Thus, $ind(w)=m<p=ind(v)$ and so we conclude that $w<v$.\\

\bigbreak

\item[(c)] Either $\sum_{i=1}^{m}k_i=\sum_{i=1}^{n}\lambda_i$, $ind(w)=ind(\sigma)$ and $i_1=j_1, \ldots , i_{s-1}=j_{s-1}, i_s>j_s$. Then,\\

\smallbreak

\noindent since $\sigma<v$, we have that either:\\

\smallbreak

$\bullet$ $\sum_{i=1}^{n}\lambda_i<\sum_{i=1}^{p}\mu_i$, same as in case (a), or\\

\smallbreak

$\bullet$ $\sum_{i=1}^{n}\lambda_i=\sum_{i=1}^{p}\mu_i$ and $ind(\sigma)<ind(v)$, same as in case (b), or\\

\smallbreak

$\bullet$ $ind(\sigma)=ind(v)$ and $j_1=\varphi_1, \ldots , j_p>\varphi_p$. Then:\\

\smallbreak

\ \ \ $(i)$ if $p=s$ we have that $i_s>j_s>\varphi_s$ and we conclude that $w<v$.\\

\smallbreak

\ \ \ $(ii)$ if $p<s$ we have that $i_p=j_p>\varphi_p$ and thus $w<v$ and if $s<p$ we have that \\

\smallbreak

\ \ \ \ \ \ \ \ \ $i_s>j_s=\varphi_s$ and so $w<v$.\\

\bigbreak

\item[(d)] Either $\sum_{i=1}^{m}k_i=\sum_{i=1}^{n}\lambda_i$, $ind(w)=ind(\sigma)$, $i_n=j_n\ \forall n$ and $k_n=\lambda_n, \ldots,
    |k_q|<|\lambda_q|$.\\

\smallbreak

\noindent Then, since $\sigma<v$, we have that either:\\

\smallbreak

$\bullet$ $\sum_{i=1}^{n}\lambda_i<\sum_{i=1}^{p}\mu_i$, same as in case (a), or\\

\smallbreak

$\bullet$ $\sum_{i=1}^{n}\lambda_i=\sum_{i=1}^{p}\mu_i$ and $ind(\sigma)<ind(v)$, same as in case (b), or\\

\smallbreak

$\bullet$ $ind(\sigma)=ind(v)$ and $j_1=\varphi_1, \ldots, j_q>\varphi_q$, same as in case (c), or\\

\smallbreak

$\bullet$ $j_n=\varphi_n$, for all $n$ and $|\mu_p| \geq |\lambda_p|$ for some $p$.\\

\smallbreak

\ \ \ $(1)$ If $|\mu_p| > |\lambda_p|$, then:\\

\smallbreak

\ \ \ \ \ \ \ \ $(i)$ If $p\geq q$ then $|k_p|=|\lambda_p|<|\mu_p|$ and thus $w<v$.\\

\smallbreak

\ \ \ \ \ \ \ \ $(ii)$ If $p<q$ then $|k_q|<|\lambda_q|=|\mu_q|$ and thus $w<v$.\\

\smallbreak

\ \ \ $(2)$ If $|\mu_p| = |\lambda_p|$, then:\\

\smallbreak

\ \ \ \ \ \ \ \ $(i)$ If $p\geq q$ then $|k_p|=|\lambda_p|=|\mu_p|$ and $k_p=\lambda_p>\mu_p$. Thus $w<v$.\\

\smallbreak

\ \ \ \ \ \ \ \ $(ii)$ If $p<q$ then $|k_q|<|\lambda_q|=|\mu_q|$ and thus $w<v$.

\end{itemize}

\bigbreak

\noindent So, we conclude that the ordering relation is transitive.

\end{proof}

\begin{defn} \label{level}
\rm
We define the subset of \textit{level $k$}, $\Lambda_k$, of $\Lambda$ to be the set
$$\Lambda_k:=\{t^{k_0}t_1^{k_1}\ldots t_{m}^{k_m} | \sum_{i=0}^{m}{k_i}=k \}$$
and similarly, the subset of \textit{level $k$} of $\Lambda^{\prime}$ to be
$$\Lambda^{\prime}_k:=\{t^{k_0}{t^{\prime}_1}^{k_1}\ldots {t^{\prime}_{m}}^{k_m} | \sum_{i=0}^{m}{k_i}=k \}.$$
\end{defn}

\begin{remark} \label{ordgap} \rm
Let $w \in \Lambda_k$ a monomial containing gaps in the indices and $u \in \Lambda_k$ a monomial with consecutive indices such that $ind(w)=ind(u)$. Then,
it follows from Definition~\ref{order} that $w<u$.
\end{remark}

\begin{prop}
The sets $\Lambda_k$ are totally ordered and well-ordered for all $k$.
\end{prop}

\begin{proof}
Since $\Lambda_k \subseteq \Lambda,\ \forall k$, $\Lambda_k$ inherits the property of being a totally ordered set from $\Lambda$. Moreover, $t^k$ is the minimum element of $\Lambda_k$ and so $\Lambda_k$ is a well-ordered set.
\end{proof}

We also introduce the notion of \textit{homologous words} as follows:

\begin{defn} \rm
We shall say that two words $w^{\prime}\in \Lambda^{\prime}$ and $w\in \Lambda$ are {\it homologous}, denoted $w^{\prime}\sim w$, if $w$ is
obtained from $w^{\prime}$ by turning $t^{\prime}_i$ into $t_i$ for all $i$.
\end{defn}

With the above notion the proof of Theorem~\ref{mainthm} is based on the following idea: Every element $w^{\prime}\in \Lambda^{\prime}$ can be expressed as linear combinations of monomials $w_i\in \Lambda$ with coefficients in $\mathbb{C}$, such that:

\begin{itemize}
\item[(i)] $\exists \ j$ such that $w_j:=w\sim w^{\prime}$,
\item[(ii)] $w_j < w_i$, for all $i \neq j$,
\item[(iii)] the coefficient of $w_j$ is an invertible element in $\mathbb{C}$.
\end{itemize}

\section{From $\Lambda^{\prime}$ to $\Sigma_n$}

In this section we prove a series of lemmas relating elements of the two different basic sets $\Sigma_n$, $\Sigma^{\prime}_n$ of $\textrm{H}_{1,n}(q)$. In the proofs we underline expressions which are crucial for the next step. Since $\Lambda^{\prime}$ is a subset of $\Sigma^{\prime}_n$, all lemmas proved here apply also to $\Lambda^{\prime}$ and will be used in the context of the bases of $\mathcal{S}({\rm ST})$.

\subsection{Some useful lemmas in $\textrm{H}_{1,n}(q)$}

We will need the following results from \cite{La2}. The first lemma gives some basic relations of the braiding generators.

\begin{lemma}[Lemma~1 \cite{La2}] \label{brrel}
For $\epsilon \in \{ \pm 1 \}$ the following hold in $\textrm{H}_{1,n}(q)$:

\noindent \textit{(i)}\ $ g_i^m \ \ \ = \ \left(q^{m-1}-q^{m-2}+ \ldots +(-1)^{m-1} \right)g_i + \left(q^{m-1}-q^{m-2}+ \cdots +(-1)^{m-2}q \right)$\\

\noindent \ \ \ \ \ $ g_i^{-m} \ = \ \left(q^{-m}-q^{1-m}+ \ldots + (-1)^{m-1}q^{-1} \right)g_i + \left(q^{-m}-q^{1-m} + \cdots +(-1)^{m-1}q^{-1}+(-1)^{m} \right)$\\

\noindent \textit{(ii)} $ {g_i}^\epsilon ({g_k}^{\pm 1}g_{k-1}^{\pm 1}\ldots {g_j}^{\pm 1}) \ = \ ({g_k}^{\pm 1}g_{k-1}^{\pm 1}\ldots {g_j}^{\pm 1}){g_{i+1}}^\epsilon, \ \mbox{ \rm for} \ k>i\geq j$, \\

\noindent \ \ \ \ \ ${g_i}^\epsilon ({g_j}^{\pm 1}g_{j+1}^{\pm 1}\ldots {g_k}^{\pm 1})\ = \ ({g_j}^{\pm 1}g_{j+1}^{\pm 1}\ldots {g_k}^{\pm1}){g_{i-1}}^\epsilon, \ \mbox{ \rm for} \ k\geq i> j$,\\

\noindent where the sign of the ${\pm 1}$ exponent is the same for all generators.\\

\noindent  \textit{(iii)} $ g_ig_{i-1}\ldots g_{j+1}{g_j} g_{j+1}\ldots g_i \ \ \ \ \ =\ g_jg_{j+1}\ldots g_{i-1}{g_i} g_{i-1}\ldots g_{j+1}{g_j} $\\

\noindent \ \ \ \ \ \  ${g_i}^{- 1}g_{i-1}^{- 1}\ldots g_{j+1}^{- 1}{g_j}^\epsilon g_{j+1}\ldots g_i \ = \ g_jg_{j+1}\ldots g_{i-1}{g_i}^\epsilon g_{i-1}^{- 1}\ldots g_{j+1}^{-1}{g_j}^{- 1} $\\

\noindent \textit{(iv)} $ {g_i}^\epsilon \ldots {g_{n-1}}^\epsilon {g_n}^{2\epsilon} {g_{n-1}}^\epsilon\ldots {g_i}^\epsilon\ \ \ = \  \sum_{r=0}^{n-i+1} \, (q^\epsilon -1)^{\epsilon_r} q^{\epsilon r} \, ({g_i}^\epsilon \ldots {g_{n-r}}^\epsilon \ldots {g_i}^\epsilon)$,\\

\noindent where $\epsilon_r=1 \ \mbox{ if } \ r\leq n-i \ \mbox{ and } \ \epsilon_{n-i+1}=0.$ Similarly,\\

\noindent \textit{(v)} $ {g_i}^\epsilon \ldots {g_2}^\epsilon {g_1}^{2\epsilon} {g_2}^\epsilon \ldots {g_i}^\epsilon \ = \ \sum_{r=0}^{i}\, (q^\epsilon -1)^{\epsilon_r} q^{\epsilon r} \, ({g_i}^\epsilon \ldots{g_{r+2}}^\epsilon {g_{r+1}}^\epsilon {g_{r+2}}^\epsilon \ldots {g_i}^\epsilon),$\\

\noindent where $\epsilon_r=1 \ \mbox{ if } \ r\leq i-1 \ \mbox{ and } \ \epsilon_i=0$.
\end{lemma}

The next lemma comprises relations between the braiding generators and the looping generator $t$.

\begin{lemma}[cf. Lemmas~1, 4, 5 \cite{La2}] \label{brlre1}
For $\epsilon \in \{ \pm 1 \}$, $i, k\in \mathbb{N}$ and $\lambda \in \mathbb{Z}$ the following hold in ${\rm H}_{1,n}(q)$:
\[
\begin{array}{llll}
(i) & t^{\lambda} g_1tg_1 & = & g_1tg_1 t^{\lambda}  \\
&&&\\
(ii) & t^\epsilon {g_1}^\epsilon t^{\epsilon k}{g_1}^\epsilon &  =  & {g_1}^\epsilon t^{\epsilon k} {g_1}^\epsilon t^\epsilon  +  (q^\epsilon -1) t^\epsilon {g_1}^\epsilon t^{\epsilon k} + (1-q^\epsilon) t^{\epsilon k}{g_1}^\epsilon t^\epsilon \\
&&&\\
& t^{-\epsilon} {g_1}^\epsilon t^{\epsilon k}{g_1}^\epsilon & = & {g_1}^\epsilon t^{\epsilon k}
{g_1}^\epsilon t^{-\epsilon}  +  (q^\epsilon -1) t^{\epsilon (k-1)} {g_1}^\epsilon +  (1-q^\epsilon) {g_1}^\epsilon t^{\epsilon (k-1)}\\
&&&\\
(iii) & t^{\epsilon i}{g_1}^{\epsilon} t^{\epsilon k}{g_1}^{\epsilon} & = & g_1^{\epsilon} t^{\epsilon k} g_1^{\epsilon} t^{\epsilon i} + (q^{\epsilon} -1)
\sum_{j=1}^{i}{t^{\epsilon j} g_1^{\epsilon} t^{\epsilon (k+i-j)}} + (1-q^{\epsilon}) \sum_{j=0}^{i-1}{t^{\epsilon (k+j)} g_1^{\epsilon} t^{\epsilon (i-j)}} \\
&&&\\
& t^{-\epsilon i}{g_1}^\epsilon t^{\epsilon k}{g_1}^\epsilon & = & {g_1}^\epsilon t^{\epsilon k}{g_1}^\epsilon t^{-\epsilon i} + (q^\epsilon -1)
\sum_{j=1}^{i}{t^{\epsilon (k-j)} g_1^{\epsilon} t^{-\epsilon (i-j)}} + (1-q^{\epsilon}) \sum_{j=1}^{i}{t^{\epsilon (i-j)} g_1^{\epsilon} t^{\epsilon (k-j)}}  \\
\end{array}
\]
\end{lemma}

The next lemma gives the interactions of the braiding generators and the loopings $t_i$s and $t^{\prime}_i$s.

\begin{lemma}[Lemmas~1 and 2 \cite{La2}] \label{brlre2}
The following relations hold in $\textrm{H}_{1,n}(q)$:

\[
\begin{array}{llll}
(i) & g_i{t_k}^\epsilon & = & {t_k}^\epsilon g_i \ \mbox{ \rm for}  \ k>i, \, k<i-1 \\
&&&\\
    & g_it_i & = & q t_{i-1}g_i+(q-1) t_i \\
&&&\\
    & g_it_{i-1} & = & q^{-1} t_ig_i+(q^{-1}-1) t_i \ =\ t_ig_i^{-1} \\
&&&\\
    & g_i{t^{-1}_{i-1}} & = & q {t_i}^{-1}g_i+(q-1) {t_{i-1}}^{-1} \\
&&&\\
    & g_i{t^{-1}_i} & = & q^{-1} {t_{i-1}}^{-1}g_i+(q^{-1}-1) {t_{i-1}}^{-1}\ = \ t_{i-1}^{-1}g_i^{-1}\\
&&&\\
(ii) & t_n^kg_n & = & (q-1)\sum_{j=0}^{k-1}{q^jt_{n-1}^jt_n^{k-j}}+q^kg_nt_{n-1}^k, \mbox{ \rm if}  \ k\in \mathbb{N} \\
&&&\\
     & t_n^kg_n & = & (1-q)\sum_{j=0}^{k-1}{q^jt_{n-1}^jt_n^{k-j}}+q^kg_nt_{n-1}^k, \mbox{ \rm if}  \ k\in \mathbb{Z} - \mathbb{N} \\
&&&\\
(iii) & {t_i}^k{t_j}^\lambda & = & {t_j}^\lambda {t_i}^k\ \mbox{ \rm for}  \ i\neq j \ \mbox{ \rm and}  \ k, \lambda\in {\mathbb Z} \\
&&&\\
(iv)  & g_i{t^{\prime}_k}^\epsilon & = & {t^{\prime}_k}^\epsilon g_i \ \ \mbox{ \rm for}  \ k>i, \ k<i-1 \\
&&&\\
      & g_i{t^{\prime}_i}^\epsilon & = & {t^{\prime}_{i-1}}^\epsilon g_i + (q-1) {t^{\prime}_i}^\epsilon + (1-q) {t^{\prime}_{i-1}}^\epsilon \\
&&&\\
      & g_i{t^{\prime}_{i-1}}^\epsilon & = & {t^{\prime}_i}^\epsilon g_i  \\
&&&\\
(v) & {t_i^{\prime}}^k & = & g_i\ldots g_1 t^k {g_1}^{-1}\ldots {g_i}^{-1} \ \mbox{ \rm for}  \ k \in {\mathbb Z}. \\
\end{array}
\]
\end{lemma}

Using now Lemmas~\ref{brrel}, \ref{brlre1} and \ref{brlre2} we prove the following relations, which we will use for converting elements in $\Lambda^{\prime}$ to elements in $\Sigma_n$. Note that whenever a generator is overlined, this means that the specific generator is omitted from the word.

\begin{lemma} \label{brlre3}
The following relations hold in ${\rm H}_{1,n}(q)$ for $k \in \mathbb{N}$:

\[
\begin{array}{llll}
(i) & g_{m+1}t_m^k & = & q^{-(k-1)}t_{m+1}^k g_{m+1}^{-1}\ +\ \sum_{j=1}^{k-1}{q^{-(k-1-j)}(q^{-1}-1)t_m^j t_{m+1}^{k-j}},\\
&&&\\
(ii) & g_{m+1}^{-1}t_m^{-k} & = & q^{(k-1)}t_{m+1}^{-k} g_{m+1}\ +\ \sum_{j=1}^{k-1}{q^{(k-1-j)}(q-1)t_m^{-j} t_{m+1}^{-(k-j)}}.
\end{array}
\]
\end{lemma}

\begin{proof}
We prove relations~(i) by induction on $k$. Relations~(ii) follow similarly. For $k=1$ we have that $g_{m+1}t_m=t_{m+1}g_{m+1}^{-1}$, which holds from Lemma~\ref{brlre2}~(i). Suppose that the relation holds for $k-1$. Then, for $k$ we have:\\

\noindent $g_{m+1}t_m^k=g_{m+1}t_m^{k-1} t_m \overset{ind.}{\underset{step}{=}} q^{-(k-2)}t_{m+1}^{k-1} \underline{g_{m+1}^{-1}}t_m\ +\ \sum_{j=1}^{k-2}{q^{-(k-2-j)}(q^{-1}-1)t_m^j t_{m+1}^{k-1-j}}t_m=$\\

\noindent $=\ q^{-(k-1)}g_{m+1}t_m\ +\ q^{-(k-2)}(q^{-1}-1)t_m t_{m+1}^{k-1} \ + \sum_{j=1}^{k-2}{q^{-(k-2-j)}(q^{-1}-1)t_m^{j+1} t_{m+1}^{k-1-j}}\ =$\\

\noindent $=\ q^{-(k-1)}t_{m+1}g_{m+1}^{-1}\ +\ \sum_{j=1}^{k}{q^{-(k-1-j)}(q^{-1}-1)t_m^{j} t_{m+1}^{k-j}}$.
\end{proof}

\begin{lemma} \label{loopcycles1}
In ${\rm H}_{1,n}(q)$ the following relations hold:

\begin{itemize}
\item[(i)] For the expression $A=\left(g_rg_{r-1} \ldots g_{r-s}\right)\cdot t_{k}$ the following hold for the different values of $k \in \mathbb{N}$:
\[
\begin{array}{llll}
(1) & A & = & t_k \left(g_r \ldots g_{r-s} \right) \ \ \ \ \ \ \ \ \ \ \ \ \ \ \ \ \ \ \ \ \ \ \ \ \ \ \ \ \ \ \ \ \ \ \ \ \ \ \ \ \ \ \ \ \ \ \ \ \ \ {\rm for}\ k>r\ {\rm or}\ k<r-s-1 \\
&&&\\
(2) & A & = & t_r \left(g_{r}^{-1} \ldots g_{r-s}^{-1} \right) \ \ \ \ \ \ \ \ \ \ \ \ \ \ \ \ \ \ \ \ \ \ \ \ \ \ \ \ \ \ \ \ \ \ \ \ \ \ \ \ \ \ \ \ \ \ \ \ {\rm for}\ k=r-s-1 \\
&&&\\
(3) & A & = & qt_{r-1} \left(g_r \ldots g_{r-s})+(q-1)t_r(g_{r-1} \ldots g_{r-s} \right) \ \ \ \ \ \qquad \ \ \ \ \ {\rm for}\ k=r \\
&&&\\
(4) & A & = & qt_{r-s-1} \left(g_{r} \ldots g_{r-s} \right)+(q-1)t_r \left(g_{r}^{-1} \ldots g_{r-s+1}^{-1} \right)  \ \ \ \ \ \ \ \ \ \ {\rm for}\ k=r-s \\
&&&\\
(5) & A & = & t_{m-1} \left(g_{r} \ldots g_{r-s} \right)+(q-1)t_r \left(g_{r}^{-1} \ldots g_{m+1}^{-1} \right)(g_{m-1}\ldots g_{r-s})\\
&&&\\
& & & {\rm for} \ k=m\in \{r-s+1,\ldots, r-1\}. \\
\end{array}
\]

\item[(ii)] For the expression $A=\left(g_rg_{r-1} \ldots g_{r-s}\right)\cdot t_{k}^{-1}$ the following hold for the different values of $k \in \mathbb{N}$:
\[
\begin{array}{llll}
(1) & A & = & t_{k}^{-1} \left(g_r \ldots g_{r-s} \right) \ \ \ \ \ \ \ \ \ \ \ \ \ \ \ \ \ \ \ \ \ \ \ \ \ \ \ \ \ \ \ \ \ \ \ \ \ \ \ \ \ \ \ {\rm for}\ k>r\ {\rm or}\ k<r-s-1 \\
&&&\\
(2) & A & = & t_{r-s-1}^{-1} \left(g_r \ldots g_{r-s+1}g_{r-s}^{-1} \right)  \ \ \ \ \ \ \ \ \ \ \ \ \ \ \ \ \ \ \ \ \ \ \ \ \ \ \ \ \ \ \  {\rm for}\ k=r-s \\
&&&\\
(3) & A & = & t_{m-1}^{-1} \left(g_rg_{r-1} \ldots g_{m+1}g_{m}^{-1}g_{m-1}\ldots g_{k-s} \right) \ \ \ \ \ \ \ \ \ \ \ \ \ \ \ {\rm for}\ k=m\in \{r-s+1, \ldots, r\}  \\
&&&\\
(4) & A & = & q^{s+1}t_{r}^{-1} \left(g_{r} \ldots g_{r-s} \right)+(q-1)\sum_{j=1}^{s+1}q^{s-j+1}t_{r-j}^{-1} \left(g_r\ldots g_{r-j+2}g_{r-j}\ldots g_{r-s}\right)  \\
&&&\\
& & &  {\rm for}\ k=r-s-1. \\
\end{array}
\]
\end{itemize}
\end{lemma}

\begin{proof}

We only prove relations~(ii) for $k=r-s-1$ by induction on $s$ (case~4). All other relations follow from Lemma~\ref{brlre2}~(i).

\smallbreak

\noindent For $s=1$ we have:

\[
\begin{array}{lll}
g_r \underline{g_{r-1}t_{r-2}^{-1}} & = &  g_r[qt_{r-1}^{-1}g_{r-1}+(q-1)t_{r-2}^{-1}]\ =\  q\underline{g_rt_{r-1}^{-1}}g_{r-1}+(q-1)g_rt_{r-2}^{-1} \\
&&\\
& = & q[qt_r^{-1}g_r+(q-1)t_{r-1}^{-1}]g_{r-1}+(q-1)t_{r-2}^{-1}g_r\\
&&\\
& = & q^2t_r^{-1}(g_rg_{r-1})+(q-1)\left[qt_{r-1}^{-1}g_{r-1}\ + \ q^0t_{r-2}^{-1}g_r \right],\\
\end{array}
\]

\bigbreak

\noindent and so the relation holds for $s=1$. Suppose that the relation holds for $s=n$. We will show that it holds for $s=n+1$. Indeed we have:\\

\noindent $(g_r \ldots g_{r-n-1})t_{r-n-2}^{-1} =  (g_r \ldots g_{r-n})(\underline{g_{r-n-1}t_{r-n-2}^{-1}}) =  (g_r \ldots g_{r-n})\left[qt_{r-n-1}^{-1}g_{r-n-1}+(q-1)t_{r-n-2}^{-1}\right] = $\\

\noindent $ = q(\underline{g_r \ldots g_{r-n}t_{r-n-1}^{-1}})g_{r-n-1} + (q-1)(g_r \ldots g_{r-n})t_{r-n-2}^{-1} \overset{\rm ind. step}{=}\ q^{n+2}t_r^{-1}(g_r \ldots g_{r-n-1})\ +$  \\

\noindent $ +\ (q-1)\sum_{j=1}^{n+1}{q^{n-j+2}t_{r-j}^{-1}(g_r\ldots g_{r-j+2}g_{r-j}\ldots g_{r-n-1})}\ +\ (q-1)t_{r-n-2}^{-1}(g_r\ldots g_{r-n})\ =$ \\

\noindent $=\ q^{n+2}t_r^{-1}(g_r \ldots g_{r-n-1}) +\ (q-1)\sum_{j=1}^{n+2}{q^{(n+1)-j+1}t_{r-j}^{-1}(g_r\ldots g_{r-j+2}g_{r-j}\ldots g_{r-n-1})}.$\\

\end{proof}

Before proceeding with the next lemma we introduce the notion of length of $w\in \textrm{H}_n(q)$. For convenience we set $\delta_{k,r}:= g_kg_{k-1}\ldots g_{r+1}g_r$ for $k>r$ and by convention we set $\delta_{k,k}:= g_k$.

\begin{defn}\label{length} \rm
We define the \textit{length} of $\delta_{k,r} \in {\rm H}_n(q)$ as $l(\delta_{k,r}):=k-r+1$ and since every element of the Iwahori-Hecke algebra of type A can be written as $\prod_{i=1}^{n-1}{\delta_{k_i,r_i}}$ so that $k_j<k_{j+1}\ \forall j$, we define the \textit{length} of an element $w\in \textrm{H}_n(q)$ as
$$l(w):=\sum_{i=1}^{n-1}{l_i(\delta_{k_i,r_i})} = \sum_{i=1}^{n-1}{k_i-r_i+1}.$$
\end{defn}

Note that $l(g_k)=l(\delta_{k,k})=k-k+1=1$.

\begin{lemma}\label{brlre4}
For $k>r$ the following relations hold in ${\rm H}_{1,n}(q)$:

$$t_k \delta_{k,r} = \sum_{i=0}^{k-r}{q^i(q-1)\delta_{k,\overline{k-i},r}t_{k-i}} + q^{l(\delta_{k,r})} \delta_{k,r} t_{r-1},$$

\noindent where $\delta_{k,\overline{k-i},r}:= g_kg_{k-1} \ldots g_{k-i+1} g_{k-i-1} \ldots g_r:= g_k \ldots \overline{g_{k-i}} \ldots g_r$.

\end{lemma}

\begin{proof}

We prove relations by induction on $k$. For $k=1$ we have that $t_1 g_1\ =\ (q-1) t_1 + q g_1 t$, which holds. Suppose that the relation
holds for $(k-1)$, then for $k$ we have:

\[
\begin{array}{lll}
t_k \delta_{k,r} & = & \underline{t_k g_k} \cdot \delta_{k-1,r} = (q-1) t_k \delta_{k-1,r} + q g_k \underline{t_{k-1} \delta_{k-1,r}} =\\
&&\\
& = & (q-1) \delta_{k-1,r} t_k + q g_k \sum_{i=0}^{k-1-r}{q^i(q-1)\delta_{k-1,\overline{k-1-i},r}t_{k-1-i}} + q^{l(\delta_{k-1,r})+1} g_k \delta_{k-1,r} t_{r-1}=\\
&&\\
& = & \sum_{i=0}^{k-r}{q^i(q-1)\delta_{k,\overline{k-1-i},r}t_{k-1-i}} + q^{l(\delta_{k,r})} \delta_{k,r} t_{r-1}.\\
\end{array}
\]

\end{proof}

\begin{lemma} \label{loopcycles2}
In ${ \rm H}_{1,n}(q)$ the following relations hold:

\begin{itemize}
\item[(i)] For the expression $A=\left(g_rg_{r+1} \ldots g_{r+s} \right)\cdot t_{k}$ the following hold for the different values of $k \in \mathbb{N}$:
\[
\begin{array}{llll}
(1) & A & = & t_k \left(g_r \ldots g_{r+s}\right) \ \ \ \ \ \ \ \ \ \ \ \ \ \ \ \ \ \ \ \ \ \ \ \ \ \ \ \ \ \ \ \ \ \ \ \ \ \ \ \ \ \ \ \ \ {\rm for}\ k\geq r+s+1 \ {\rm or} \ k<r-1 \\
&&&\\
(2) & A & = & t_{k+1} \left(g_{r} \ldots g_{k} g_{k+1}^{-1} g_{k+2} \ldots g_{r+s} \right) \ \ \ \ \ \ \ \ \ \ \ \ \ \ \ \ \ \ \ \ \ \ \ \ \ \ \ \ \ \ \ \ \ \ \ {\rm for}\ r-1 \leq k < r+s \\
&&&\\
(3) & A & = & (q-1) \sum_{i=r}^{r+s}{q^{r+s-i}t_i \left(g_r\ldots \overline{g_i} \ldots g_{r+s} \right)} + q^{s+1}t_{r-1} \left(g_r\ldots g_{r+s} \right) \ \ {\rm for} \ \ \ \ \ \ \ k=r+s \\
\end{array}
\]

\item[(ii)] For the expression $A=\left(g_rg_{r+1} \ldots g_{r+s} \right)\cdot t_{k}^{-1}$ the following hold for the different values of $k \in \mathbb{N}$:
\[
\begin{array}{llll}
(1) & A & = & t_{k}^{-1}\left(g_rg_{r+1} \ldots g_{r+s} \right) \ \ \ \ \ \ \ \ \ \ \ \ \ \ \ \ \ \ \ \ \ \ \ \ \ \ \ \ \ \ \ \ \ \ \ \ \ \ \ {\rm for}\ k \geq r+s+1\ {\rm or} \ k < r-1 \\
&&&\\
(2) & A & = & q\ t_{k+1}^{-1} \left(g_r \ldots g_{r+s}\right) + (q-1)\ t_{r-1}^{-1} \left(g_{r}^{-1} \ldots g_k^{-1} g_{k+2}\ldots g_{r+s} \right)\ \ \ {\rm for}\ r-1 \leq k < r+s \\
&&&\\
(3) & A & = & t_{r-1}^{-1} \left(g_r^{-1} \ldots g_{r+s}^{-1} \right)\ \ \ \ \ \ \ \ \ \ \ \ \ \ \ \ \ \ \ \  \ \ \ \ \ \ \ \ \ \ \ \ \ \ \ \ \ \  \ \ \ \ \ \ \ \ \ \ \ \ \ \ \ \ \ \ \ \ \ \ \ {\rm for}\ k=r+s \\
\end{array}
\]
\end{itemize}
\end{lemma}

\begin{proof}

We prove relation~(i) for $r+s=k$ by induction on $k$ (case~3). All other relations follow from Lemmas~\ref{brrel} and \ref{brlre2}.

For $k=1$ we have: $g_1t_1 = \underline{g_1^2}tg_1 = qtg_1+(q-1)t_1$. Suppose that the relation holds for $k=n$. Then, for $k=n+1$ we have that:

\noindent $g_r\ldots \underline{g_{n+1} t_{n+1}} = q(\underline{g_r \ldots g_n t_n}) g_{n+1} + (q-1) \underline{(g_r\ldots g_n)t_{n+1}} \overset{ind. step}{=}$\\

\noindent $=q\left[(q-1) \sum_{i=r}^{n}{q^{n-i}t_i(g_r \ldots \overline{g_i} \ldots g_n)} + q^{n-r+1}t_{r-1}(g_r\ldots g_n) \right]g_{n+1} + (q-1) t_{n+1}(g_r\ldots g_n)=$\\

\noindent $= \left( (q-1) \sum_{i=r}^{n}{q^{n-i+1}t_i(g_r \ldots \overline{g_i} \ldots g_ng_{n+1})} + (q-1) t_{n+1}(g_r\ldots g_n) \right) + q^{n+1-r+1}t_{r-1}(g_r\ldots g_ng_{n+1})= $\\

\noindent $= (q-1) \sum_{i=r}^{n+1}{q^{n+1-i}t_i(g_r \ldots \overline{g_i} \ldots g_{n+1})} + q^{n+1-r+1}t_{r-1}(g_r \ldots g_{n+1}).$

\end{proof}

\begin{lemma} \label{loopbridge}
The following relations hold in ${ \rm H}_{1,n}(q)$ for $k \in \mathbb{N}$:

\[
\begin{array}{llll}
(i) & \left(g_1\ldots g_{i-1}g_i^{2}g_{i-1}\ldots g_1\right)\cdot t & = & (q-1) \sum_{k=1}^{i}{q^{i-k} t_k \left(g_1\ldots g_{k-1}g_k^{-1}g_{k-1}^{-1}\ldots g_1^{-1}\right)}+q^it \\
&&&\\
(ii) & \left(g_1^{-1} \ldots g_{i-1}^{-1} g_i^{-2} g_{i-1}^{-1} \ldots g_1^{-1}\right) \cdot t^{-1} & = & (q^{-1}-1) \sum_{k=1}^{i}{q^{-(i-k)} t_k^{-1} \left(g_1^{-1} \ldots g_{k-1}^{-1} g_k g_{k-1} \ldots g_1\right)} + \\
&&&\\
& & + & q^{-i}\ t^{-1} \\
&&&\\
(iii) & \left(g_k^{-1} \ldots g_2^{-1}g_1^{-2}g_2^{-1} \ldots g_k^{-1}\right)\cdot t_k & = & (q^{-1}-1) \sum_{i=1}^{k-1}{q^{-k}t_i \left(g_k^{-1} \ldots g_{i+2}^{-1}g_{i+1}g_{i+2} \ldots g_k \right)}\ + \\
&&&\\
& & + & q^{-k}t_k \\
&&&\\
(iv) & \left(g_k^{-1} \ldots g_2^{-1}g_1^{-2}g_2^{-1} \ldots g_k^{-1}\right)\cdot t_k^{-1} & = & t^{-1} q^{-k}(q^{-1}-1)g_k^{-1}\ldots g_1^{-1}\ldots g_k^{-1}\ + \\
&&&\\
& & + & \sum_{i=0}^{k-1}{t_i^{-1}q^{-k+i}(q^{-1}-1)g_k^{-1}\ldots g_1^{-2}\ldots g_i^{-1}g_{i+2}^{-1}\ldots g_k^{-1}}\ + \\
&&&\\
& & + &  t_k^{-1} \big[\sum_{i=2}^{k}{q^{-k+i}(q^{-1}-1)^2g_{i-1}^{-1}\ldots g_2^{-1}g_1^{-2}g_2^{-1}\ldots g_{i-1}^{-1}}\ +\\
&&&\\
& & + &  q^{-(k+1)}(q^{2}-q+1)\big].\\
\end{array}
\]
\end{lemma}

\begin{proof}

We prove relations~(i) by induction on $i$. All other relations follow similarly. For $i=1$ we have:
$g_1^2t=g_1g_1tg_1g_1^{-1}=g_1t_1g_1^{-1}=(q-1)t_1g_1^{-1}+qt$. Suppose that the relation holds for $i=n$. Then, for $i=n+1$ we have:

\smallbreak

\noindent $\left(g_1\ldots g_{n}g_{n+1}^{2}g_{n}\ldots g_1\right)\cdot t \ = \ (q-1)\left(g_1\ldots g_{n+1}g_n\ldots g_1\right)\cdot t\ + \ q\left(g_1\ldots g_{n-1}g_n^{2}g_{n-1}\ldots g_1\right)\cdot t \ =$\\

\noindent $=\ (q-1)g_1\ldots g_nt_{n+1}g_{n+1}^{-1}\ldots g_1^{-1}\ +\ q\sum_{k=1}^{n}{q^{n-k}(q-1)t_k\left(g_1\ldots g_{k-1}g_k^{-1}\ldots g_1^{-1}\right)}+q^{n+1}t\ =$\\

\noindent $=\ (q-1)t_{n+1}\left(g_1\ldots g_ng_{n+1}^{-1}\ldots g_1^{-1}\right)\ +\ \sum_{k=1}^{n}{q^{n+1-k}(q-1)t_k\left(g_1\ldots g_{k-1}g_k^{-1}\ldots g_1^{-1}\right)}+q^{n+1}t\ =$\\

\noindent $=\ \sum_{k=1}^{n+1}q^{n+1-k}(q-1)t_k\left(g_1\ldots g_{k-1}g_k^{-1}\ldots g_1^{-1}\right)+q^{n+1}t.$

\end{proof}

\subsection{Converting elements in $\Lambda^{\prime}$ to elements in $\Sigma_n$}

We are now in the position to prove a set of relations converting monomials of $t^{\prime}_i$'s to expressions containing the $t_i$'s. In \cite{D} we provide lemmas converting monomials of $t_i$'s to monomials of $t^{\prime}_i$'s in the context of giving a simple proof that the sets $\Sigma_n^{\prime}$ form bases of $\textrm{H}_{1,n}(q)$.

\begin{lemma} \label{tpr1}
The following relations hold in ${\rm H}_{1,n}(q)$ for $k \in \mathbb{N}$:
\[
\begin{array}{llll}
(i) & {t_1^{\prime}}^{-k} & = & q^k t_1^{-k} \ + \ \sum_{j=1}^{k}{q^{k-j}(q-1) t^{-j}t_1^{j-k}\cdot g_1^{-1}}, \\
&&&\\
(ii) & {t_1^{\prime}}^{k} & = & q^{-k} t_1^{k} \ + \ \sum_{j=1}^{k}{q^{-(k-j)}(q^{-1}-1) t^{j-1}t_1^{k+1-j}\cdot g_1^{-1}}. \\
\end{array}
\]
\end{lemma}

\begin{proof}
We prove relations~(i) by induction on $k$. Relations~(ii) follow similarly. For $k=1$ we have: $ {t^{\prime}_1}^{-1}  =  \underline{g_1}\ t^{-1}\ g_1^{-1}\  =\  q\ \underline{g_1^{-1}\ t^{-1}\ g_1^{-1}}\ +\ (q-1)\ t^{-1}\ g_1^{-1}\ =\  q\ t_1^{-1}\ +\ (q-1)\ t^{-1}\ g_1^{-1}.$

\smallbreak

\noindent Suppose that the relation holds for $k-1$. Then, for $k$ we have:

\[
\begin{array}{lll}
{t_1^{\prime}}^{-k} & = & {t_1^{\prime}}^{-(k-1)} {t_1^{\prime}}^{-1}\ \overset{ind.}{{\underset{step}{=}}}\ q^{k-1}t_{1}^{-(k-1)}{t_1^{\prime}}^{-1} + \sum_{j=1}^{k-1}q^{k-1-j}(q-1)t^{-j}t_1^{j-(k-1)}g_1^{-1}{t_1^{\prime}}^{-1}\\
&&\\
& = & q^k t_1^{-k} + q^{k-1} t^{-1} t_1^{-(k-1)} g_1^{-1} + \sum_{j=1}^{k-1}q^{k-1-j}(q-1)t^{-j}t_1^{j-(k-1)}t^{-1}g_1^{-1}\\
&&\\
& = & q^k t_1^{-k} + q^{k-1}(q-1) t^{-1} t_1^{-(k-1)} g_1^{-1} + \sum_{j=1}^{k-1}q^{k-1-j}(q-1)t^{-j-1}t_1^{j-(k-1)}g_1^{-1}\\
&&\\
& = & q^k t_1^{-k} + \sum_{j=1}^{k}q^{k-j}(q-1)t^{-j}t_1^{j-k}g_1^{-1}.\\
\end{array}
\]
\end{proof}

\begin{lemma} \label{tprneg}
The following relations hold in ${\rm H}_{1,n}(q)$ for $k \in \mathbb{N}$:
$$
{t^{\prime}_k}^{-1}\ =\ q^k\ t_k^{-1}\ +\ (q-1)\ \sum_{i=0}^{k-1}{q^i\
t_i^{-1}\ (\ g_k\ g_{k-1}\ \ldots\ g_{i+2}\ g_{i+1}^{-1}\ \ldots \ g_{k-1}^{-1}\ g_k^{-1}\ )}.
$$
\end{lemma}

\begin{proof}

We prove the relations by induction on $k$. For $k=1$ we have:

\smallbreak

\noindent $ {t^{\prime}_1}^{-1}  =  \underline{g_1}\ t^{-1}\ g_1^{-1}\  =\  q\ \underline{g_1^{-1}\ t^{-1}\ g_1^{-1}}\ +\ (q-1)\ t^{-1}\ g_1^{-1}\ =\  q\ t_1^{-1}\ +\ (q-1)\ t^{-1}\ g_1^{-1}$.

\bigbreak

\noindent Suppose that the relations hold for $k=n$. Then, for $k=n+1$ we have that:

\noindent ${t^{\prime}_{n+1}}^{-1} =  g_{n+1}\ \underline{{t^{\prime}_n}^{-1}}\ g_{n+1}^{-1} \overset{ind.\ step}{=} g_{n+1} \big[q^n t_n^{-1}\ +\ (q-1) \sum_{i=0}^{n-1}{q^i\ t_i^{-1} ( g_n \ldots g_{i+2} g_{i+1}^{-1} \ldots g_n^{-1} )} \big] g_{n+1}^{-1} = $\\

\noindent $ = \ q^n\ \underline{g_{n+1}\ t_n^{-1}}\ g_{n+1}^{-1}\ +\ (q-1) \sum_{i=0}^{n-1}{q^i \underline{g_{n+1} t_i^{-1}} ( g_n \ldots g_{i+2} g_{i+1}^{-1} \ldots g_n^{-1} g_{n+1}^{-1})}\ = $\\

\noindent $ =\ q^n \big[q t_{n+1}^{-1} g_{n+1} \ +\ (q-1) t_n^{-1} \big] g_{n+1}^{-1}\ +\ (q-1) \sum_{i=0}^{n-1}{q^i t_i^{-1} ( g_{n+1} \ldots
g_{i+2} g_{i+1}^{-1} \ldots g_{n+1}^{-1})}\ = $ \\

\noindent $= \ q^{n+1} t_{n+1}^{-1}\ +\ q^{n} (q-1) t_{n}^{-1} g_{n+1}^{-1}\ +\ (q-1) \sum_{i=0}^{n-1}{q^i t_i^{-1} ( g_{n+1} \ldots g_{i+2} g_{i+1}^{-1} \ldots g_{n+1}^{-1})}\ =$ \\

\noindent $ = \ q^{n+1} t_{n+1}^{-1}\ +\ (q-1) \sum_{i=0}^{n}{q^i t_i^{-1}\ ( g_{n+1} \ldots\ g_{i+2} g_{i+1}^{-1} \ldots  g_{n+1}^{-1} )}.$

\end{proof}

\begin{lemma}\label{loops}
The following relations hold in ${\rm H}_{1,n}(q)$ for $k \in \mathbb{Z}\backslash \{0 \}$:

$$ {t_m^{\prime}}^{k}\ =\ q^{-m k}t_{m}^{k} \ +\ \sum_{i}{f_i(q) t_m^{k} w_i} \ +\ \sum_{i}{g_i(q) t^{\lambda_0}t_1^{\lambda_1}\ldots t_m^{\lambda_m}u_i}, $$

\noindent where $w_i, u_i \in {\rm H}_{m+1}(q),\ \forall i$, $\sum_{i=0}^{m}{\lambda_i}=k$ and $\lambda_i \geq 0$, if $k_m>0$ and $\lambda_i \leq 0$, if $k_m<0$.

\end{lemma}

\begin{proof}
We prove relations by induction on $m$. The case $m=1$ is Lemma~\ref{tpr1}. Suppose now that the relations hold for $m-1$. Then, for $m$ we have:\\

\noindent $ {t_m^{\prime}}^{k}\ =\ g_m {t_{m-1}^{\prime}} g_m^{-1}\ \overset{ind.}{\underset{step}{=}}\ q^{-(m-1) k}\underline{g_m t_{m-1}^{k}} g_m^{-1} \ + \sum_{i}{f_i(q) \underline{g_m t_{m-1}^{k}} w_i g_m^{-1}}\ +$\\

\noindent $ +\ \sum_{i}{g_i(q) t^{\lambda_0}t_1^{\lambda_1}\ldots t_{m-2}^{\lambda_{m-2}} \underline{g_m t_{m-1}^{\lambda_{m-1}}} u_i } g_m^{-1} \ \overset{(L.4)}{=}$\\

\noindent $=\ q^{-(m-1)k}q^{-(k-1)}t_m^k \underline{g_m^{-2}}\ +\ \sum_{j=1}^{k-1}{q^{-(k-1-j)}(q^{-1}-1)t_{m-1}^jt_m^{k-j}g_m^{-1}}\ +\ \sum \ +\ \sum \ =$\\

\noindent $=\ q^{-m k}t_{m}^{k} \ +\ \sum_{i}{f_i(q) t_m^{k} w_i} \ +\ \sum_{i}{g_i(q) t^{\lambda_0}t_1^{\lambda_1}\ldots t_m^{\lambda_m}u_i}$.

\end{proof}

\noindent Using now Lemma~\ref{loops} we have that every element $u \in \Lambda^{\prime}$ can be expressed to linear combinations of elements $v_i \in \Sigma_n$, where $\exists\ j : v_j \sim u$. More precisely:

\begin{thm} \label{convert}
The following relations hold in ${\rm H}_{1,n}(q)$ for $k \in \mathbb{Z}$:
$$
t^{k_0}{t_1^{\prime}}^{k_1} \ldots {t_m^{\prime}}^{k_m} \ = \ q^{- \sum_{n=1}^{m}{nk_n}}\cdot t^{k_0}t_1^{k_1}\ldots t_m^{k_m}\ + \ \sum_{i}{f_i(q)\cdot t^{k_0}t_1^{k_1}\ldots t_m^{k_m}\cdot w_i} \ + \ \sum_{j}{g_j(q)\tau_j \cdot u_j},
$$

\noindent where $w_i, u_j \in {\rm H}_{m+1}(q), \forall i$, $\tau_j \in \Lambda$, such that $\tau_j < t^{k_0}t_1^{k_1}\ldots t_m^{k_m}, \forall i$.
\end{thm}

\begin{proof}
We prove relations by induction on $m$. Let $k_1 \in \mathbb{N}$, then for $m=1$ we have:\\

\noindent $t^{k_0}{t_1^{\prime}}^{k_1}\ \overset{(L.9)}{=}\ q^{-k_1}t^{k_0}t_1^{k_1}\ +\ \sum_{j=1}^{k_1}{q^{-(k_1-j)}(q^{-1}-1)t^{k_0+j-1}t_1^{k_1+1-j}g_1^{-1}}\ =$\\

\noindent $=\ q^{-k_1}t^{k_0}t_1^{k_1}\ +\ q^{-k_1}(q^{-1}-1) t^{k_0}t_1^{k_1}g_1^{-1} \ + \sum_{j=2}^{k_1}{q^{-(k_1-j)}(q^{-1}-1)t^{k_0+j-1}t_1^{k_1+1-j}g_1^{-1}}$.\\

\noindent On the right hand side we obtain a term which is the homologous word of $t^{k_0}{t_1^{\prime}}^{k_1}$ with scalar $q^{-k_1}\in \mathbb{C}$,
the homologous word again followed by $g_1^{-2}\in \textrm{H}_2(q)$ and with scalar $q^{-(k_1-1)}(q^{-1}-1) \in \mathbb{C}$ and the terms $t^{k_0+j-1}t_1^{k_1+1-j}$, which are of less order than the homologous word $t^{k_0}t_1^{k_1}$, since $k_1 > k_1+1-j$,\ for all $j \in \{2, 3, \ldots k_1 \}$. So the statement holds for $m=1$ and $k_1 \in \mathbb{N}$. The case $m=1$ and $k_1 \in \mathbb{Z} \backslash \mathbb{N}$ is similar.

\smallbreak

\noindent Suppose now that the relations hold for $m-1$. Then, for $m$ we have:

\[
\begin{array}{lcl}
t^{k_0}{t_1^{\prime}}^{k_1} \ldots {t_m^{\prime}}^{k_m} &  \overset{ind.}{\underset{step}{=}} & q^{- \sum_{n=1}^{m-1}{nk_n}}\cdot t^{k_0} \ldots t_{m-1}^{k_{m-1}}\cdot {t_m^{\prime}}^{k_m}\ + \ \sum_{i}{f_i(q)\cdot t^{k_0}t_1^{k_1}\ldots t_{m-1}^{k_{m-1}}\cdot w_i \cdot {t_m^{\prime}}^{k_m}}\\
&&\\
&+& \sum_{j}{g_j(q)\tau_j \cdot u_j \cdot {t_m^{\prime}}^{k_m}}.\\
\end{array}
\]

\noindent Now, since $w_i, u_i \in \textrm{H}_{m}(q),\ \forall i$ we have that $w_i {t_m^{\prime}}^{k_m}\ =\ {t_m^{\prime}}^{k_m} w_i$ and $u_i {t_m^{\prime}}^{k_m}\ = \ {t_m^{\prime}}^{k_m} u_i, \ \forall i$. Applying now Lemma~\ref{loops} to ${t_m^{\prime}}^{k}$ we obtain the requested relation.
\end{proof}

\begin{example}\label{eg1}\rm
We convert the monomial $t^{-1}{t_{1}^{\prime}}^{2}{t_2^{\prime}}^{-1} \in \Lambda^{\prime}$ to linear combination of elements in $\Sigma_n$. We have that:

\[
\begin{array}{lllr}
{t_{1}^{\prime}}^{2} & = & q^{-2}t_1^2+q^{-1}(q^{-1}-1)t_1^2g_1^{-1}+(q^{-1}-1)tt_1g_1^{-1},& {\rm (Lemma~\ref{tpr1}),}\\
&&&\\
{t_2^{\prime}}^{-1} & = & q^2t_2^{-1}+q(q-1)t^{-1}g_2^{-1}g_1^{-1}g_2^{-1}+q(q-1)t_1^{-1}g_2^{-1}+(q-1)^2t^{-1}g_1^{-1}g_2^{-1},& {\rm (Lemma~\ref{tprneg}),}\\
\end{array}
\]

\noindent and so:

\[
\begin{array}{lll}
t^{-1}{t_{1}^{\prime}}^{2}{t_2^{\prime}}^{-1} & = & t^{-1}t_{1}^{2}t_2^{-1}\cdot \left(1+ q^2(q^{-1}-1)g_1^{-1} \right)\ +\ t^{-2}t_1^{2} \cdot \left(q^{-1}(q-1)  g_2^{-1}g_1^{-1}g_2^{-1}\right) \ +\\
&&\\
& + & t^{-1}t_1 \cdot \left(q^{-1}(q-1)g_2^{-1}\ +\ (q-1)(q^{-1}-1)g_2^{-1}g_1^{-1}\ +\ (q-1)(q^{-1}-1)g_1^{-1}g_2^{-1} \right)\ +\\
&&\\
& + & 1 \cdot \left( -(q-1)^2 g_2^{-1}g_1^{-1} \right)\ +\ t_1t_2^{-1} \left( q^2(q^{-1}-1) g_1^{-1} \right).\\
\end{array}
\]

We obtain the homologous word $w=t^{-1}t_{1}^{2}t_2^{-1}$, the homologous word again followed by the braiding generator $g_1^{-1}$ and all other terms are of less order than $w$ since, either they contain gaps in the indices such as the term $t_1t_2^{-1}$, or their index is less than $ind(w)$ (the terms $t^{-1}t_1$, $t^{-2}t_1^{2}$, $1$).
\end{example}

\section{ From $\Sigma_n$ to $\Lambda$}

\subsection{Managing the gaps}

Before proceeding with the proof of Theorem~\ref{mainthm} we need to discuss the following situation. According to Lemma~\ref{tpr1}, for a word $w^{\prime}=t^{k}{t^{\prime}_1}^{-\lambda}\in\Lambda^{\prime}$, where $k,\lambda \in \mathbb{N}$ and $k<\lambda$ we have that:

\[
\begin{array}{lllll}
w^{\prime} & = & t^{k}{t^{\prime}_1}^{-\lambda} & = & t^{k-1}{t_1}^{-\lambda+1}\alpha_1+t^{k-2}{t_1}^{-\lambda+2}\alpha_2\ +\ \ldots \ +\ t^{k-(k-1)}{t_1}^{-\lambda+(k-1)}\alpha_{k-1}\ +\\
&&&&\\
& & & + & t^0{t_1}^{-\lambda+k}\alpha_{k}\ +\ t^{-1}{t_1}^{-\lambda+k+1}\alpha_{k+1}\ +\ \ldots \ +\ t^{-\lambda+k}\alpha_{\lambda},\\
\end{array}
\]

\noindent where $\alpha_i\in \textrm{H}_n(q),\ \forall i$. We observe that in this particular case, in the right hand side there are terms which do not belong to the set
$\Lambda$. These are the terms of the form $t_1^m$. So these elements cannot be compared with the highest order term $w\sim w^{\prime}$. The point now is that a
term $t_1^m$ is an element of the basis $\Sigma_n$ on the Hecke algebra level, but, when we are working in $\mathcal{S}({\rm ST})$, such an element must be considered up to conjugation by any braiding generator and up to stabilization moves. Topologically, conjugation corresponds to closing the braiding part of a mixed braid.
Conjugating $t_1$ by $g_1^{-1}$ we obtain $tg_1^{2}$ (view Figure~\ref{conj2}) and similarly conjugating $t_1^m$ by $g_1^{-1}$ we obtain $tg_1^2tg_1^2\ldots tg_1^2$. Then, applying Lemma~\ref{brlre2} we obtain the expression $\sum_{k=1}^{m-1}{t^kt_1^{m-k}}v_k$, where $v_k\in \textrm{H}_n(q)$, for all $k$, that is, we obtain now elements in the $\bigcup_n\textrm{H}_n(q)$-module $\Lambda$.

\begin{figure}
\begin{center}
\includegraphics[width=4.7in]{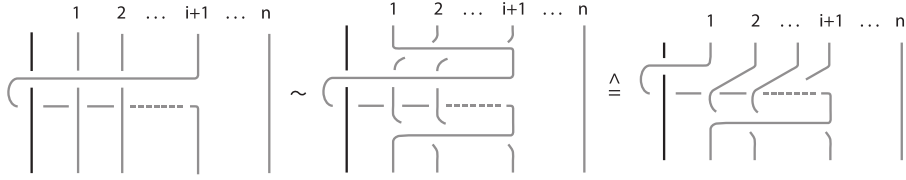}
\end{center}
\caption{Conjugating $t_i$ by $g_1^{-1}\ldots g_{i}^{-1}$.}
\label{conj2}
\end{figure}

\smallbreak

We shall next treat this situation in general. For the expressions that we obtain after appropriate conjugations we shall use the notation $\widehat{=}$. We will call {\it gaps} in monomials of the $t_i$'s, gaps occurring in the indices and \textit{size} of the gap $t_i^{k_i}t_j^{k_j}$ the number $s_{i,j}=j-i \in \mathbb{N}$.

\begin{lemma}\label{gapsimple}
For $k_0,k_1 \ldots k_i \in \mathbb{Z}$, $\epsilon = 1$ or $\epsilon = -1$ and $s_{i,j}>1$ the following relation holds in ${\rm H}_{1,n}(q)$:

$$ t^{k_0}t_1^{k_1}\ldots t_{i-1}^{k_{i-1}}t_i^{k_i}\cdot t^{\epsilon}_j \ \widehat{=} \ t^{k_0}t_1^{k_1}\ldots t_{i-1}^{k_{i-1}}t_i^{k_i}\cdot t^{\epsilon}_{i+1} \left(g^{\epsilon}_{i+2}\ldots g^{\epsilon}_{j-1}g_j^{2\epsilon}g^{\epsilon}_{j-1} \ldots g^{\epsilon}_{i+2} \right).$$

\end{lemma}

\begin{proof}

We have that $t_j^{\epsilon}\ =\ \left(g_j^{\epsilon} \ldots g_{i+2}^{\epsilon} \right)\ t_{i+1}^{\epsilon}\ \left(g_{i+2}^{\epsilon} \ldots
g_{j}^{\epsilon} \right)$ and so:

\[
\begin{array}{lcll}
t^{k_0}t_1^{k_1}\ldots t_{i-1}^{k_{i-1}}t_i^{k_i}t_j^{\epsilon} & = & t^{k_0}t_1^{k_1}\ldots t_{i-1}^{k_{i-1}}t_i^{k_i} (g_j^{\epsilon} \ldots
 g_{i+2}^{\epsilon})\ t_{i+1}^{\epsilon}\ (g_{i+2}^{\epsilon} \ldots g_{j}^{\epsilon}) & =\\
&&&\\
& = & (g_j^{\epsilon} \ldots g_{i+2}^{\epsilon})\ t^{k_0}t_1^{k_1}\ldots t_{i-1}^{k_{i-1}}t_i^{k_i}t_{i+1}^{\epsilon} (g_{i+2}^{\epsilon} \ldots
g_{j}^{\epsilon}) & \widehat{=}\\
&&&\\
& \widehat{=} & t^{k_0}\ldots t_{i-1}^{k_{i-1}}t_i^{k_i}t_{i+1}^{\epsilon} (g_{i+2}^{\epsilon}\ldots
g_{j-1}^{\epsilon}g_j^{2{\epsilon}}g_{j-1}^{\epsilon} \ldots g_{i+2}^{\epsilon} ).&\\
\end{array}
\]

\end{proof}

In order to pass to a general way for managing gaps in monomials of $t_i$'s we first deal with gaps of size one. For this we have the following.

\begin{lemma}\label{conj}
For $k \in \mathbb{N}$, $\epsilon = 1$ or $\epsilon = -1$ and $\alpha \in {\rm H}_{1,n}(q)$ the following relations hold:

$$t_i^{\epsilon k} \cdot \alpha\ \widehat{=}\ \sum_{u=1}^{k-1}{q^{\epsilon (u-1)}(q^{\epsilon}-1)t_{i-1}^{\epsilon u}t_i^{\epsilon (k-u)} (\alpha g_i^{\epsilon})}\ +\ q^{\epsilon (k-1)}t_{i-1}^{\epsilon k} ( g_i^{\epsilon} \alpha g_i^{\epsilon}). $$
\end{lemma}

\begin{proof}

We prove the relations by induction on $k$. For $k=1$ we have $t_i^{\epsilon}\cdot \alpha \ \widehat{=}\ g_i^{\epsilon}t_{i-1}^{\epsilon} g_i^{\epsilon} \cdot \alpha \ \widehat{=}\ t_{i-1}^{\epsilon} g_i^{\epsilon} \cdot \alpha \cdot g_i^{\epsilon}$. Suppose that the assumption holds for $k-1>1$. Then for $k$ we have:

\smallbreak

\noindent $t_i^{\epsilon k} \cdot \alpha\ \widehat{=}\ t_i^{\epsilon (k-1)} (t_i^{\epsilon} \cdot \alpha) \ \overset{(t_i^{\epsilon}\cdot \alpha\ =\ \beta)}{=}\ t_i^{\epsilon (k-1)} \cdot \beta\ \underset{ind.\ step}{\widehat{=}}$

\noindent $=\  \sum_{u=1}^{k-2}{q^{\epsilon (u-1)}(q^{\epsilon}-1)t_{i-1}^{\epsilon u}t_i^{\epsilon (k-1-u)} (\beta g_i^{\epsilon})}\ +\ q^{\epsilon (k-2)}t_{i-1}^{\epsilon (k-1)} ( g_i^{\epsilon} \beta g_i^{\epsilon})\ \overset{(\beta\ =\ t_i^{\epsilon}\cdot \alpha)}{=}$\\

\noindent $=\ \sum_{u=1}^{k-2}{q^{\epsilon (u-1)}(q^{\epsilon}-1)t_{i-1}^{\epsilon u}t_i^{\epsilon (k-1-u)}t_i^{\epsilon} (\alpha g_i^{\epsilon})}\ +\ q^{\epsilon (k-2)}t_{i-1}^{\epsilon (k-1)} ( g_i^{\epsilon}t_i^{\epsilon} \alpha g_i^{\epsilon})\ =$\\

\noindent $=\ \sum_{u=1}^{k-2}{q^{\epsilon (u-1)}(q^{\epsilon}-1)t_{i-1}^{\epsilon u}t_i^{\epsilon (k-u)} (\alpha g_i^{\epsilon})}\ +\ q^{\epsilon (k-2)}t_{i-1}^{\epsilon (k-1)} t_i^{\epsilon} \alpha g_i^{\epsilon}\ +\ q^{\epsilon (k-1)}t_{i-1}^{\epsilon (k-1+1)} ( g_i^{\epsilon}t_i^{\epsilon} \alpha g_i^{\epsilon})\ =$\\

\noindent $=\ \sum_{u=1}^{k-1}{q^{\epsilon (u-1)}(q^{\epsilon}-1)t_{i-1}^{\epsilon u}t_i^{\epsilon (k-u)} (\alpha g_i^{\epsilon})}\ +\ q^{\epsilon (k-1)}t_{i-1}^{\epsilon k} ( g_i^{\epsilon} \alpha g_i^{\epsilon}) $.
\end{proof}

We now introduce the following notation.

\begin{nt}\label{nt} \rm
We set $\tau_{i,i+m}^{k_{i,i+m}}:=t_i^{k_i}t^{k_{i+1}}_{i+1}\ldots t^{k_{i+m}}_{i+m}$, where $m\in \mathbb{N}$ and $k_j\neq 0$ for all $j$ and

\[
\delta_{i,j}:=\left\{\begin{matrix}
g_ig_{i+1}\ldots g_{j-1}g_{j} & if \ i<j \\
&\\
g_ig_{i-1}\ldots g_{j+1}g_{j} & if \ i>j
\end{matrix}\right. ,\ \
\delta_{i,\widehat{k},j}:=\left\{\begin{matrix}
g_ig_{i+1} \ldots g_{k-1}g_{k+1} \ldots g_{j-1}g_{j} & if \ i<j \\
&\\
g_ig_{i-1} \ldots g_{k+1}g_{k-1} \ldots g_{j+1}g_{j} & if \ i>j
\end{matrix}\right.
\]

\noindent We also set $w_{i,j}$ an element in $\textrm{H}_{j+1}(q)$ where the minimum index in $w$ is $i$.
\end{nt}

Using now the notation introduced above, we apply Lemma~\ref{conj} $s_{i,j}$-times to 1-gap monomials of the form $\tau_{0,i}^{k_{0,i}}\cdot t_j^{k_j}$ and we obtain monomials with no gaps in the indices, followed by words in ${\rm H}_n(q)$.

\begin{example} For $s_{i,j}>1$ and $\alpha \in {\rm H}_n(q)$ we have:

\[
\begin{array}{lrcl}
(i) & \tau_{0,i}^{k_i}\cdot t_j\cdot \alpha & \widehat{=} & \tau_{0,i}^{k_i}\cdot t_{i+1}\cdot \delta_{i+2,j}\ \alpha\ \delta_{j,i+2}\\
&&&\\
(ii) & \tau_{0,i}^{k_i}\cdot t^{2}_j\cdot \alpha & \widehat{=} & \tau_{0,i}^{k_i}\cdot t^{2}_{i+1}\cdot \delta_{i+2,j}\ \alpha\ \delta_{j,i+2}\ +\ \tau_{0,i}^{k_i}\cdot t_{i+1}t_{i+2}\cdot \beta,\ \rm{where}\\
&&&\\
& \beta & = & \left[(q-1)\sum_{s=i+2}^{j}{q^{j-s} \delta_{i+3,s} \delta_{i+2,s-1} \delta_{s+1,j}}\ \alpha\ \delta_{j,i+2}\delta_{s,i+3}\right]\\
&&&\\
(iii) & \tau_{0,i}^{k_i}\cdot t^{3}_j\cdot \alpha & \widehat{=} & \left[q^{j-(i+2)+1} \right]^2 \tau_{0,i}^{k_i}\cdot t^{3}_{i+1}\cdot \delta_{i+2,j}\ \alpha\ \delta_{j,i+2}\ +\ \tau_{0,i}^{k_i}\cdot t^{2}_{i+1}t_{i+2}\cdot \beta\ +\\
&&&\\
&& + &  \tau_{0,i}^{k_i}\cdot t_{i+1}t^{2}_{i+2}\cdot \gamma\ +\ \tau_{0,i}^{k_i}\cdot t_{i+1}t_{i+2}t_{i+3}\cdot \mu,\ \rm{where}\\
&&&\\
&\gamma & = & q^{j-(i+3)+1}(q-1) \delta_{i+3,j} \delta_{i+2,s-1} \delta_{s+1,j}\ \alpha\ \delta_{j,i+2} \delta_{s,i+3},\ \rm{and}\\
&&&\\
& \mu & = & \sum_{s=i+2}^{j} \sum_{r=s+1}^{j}q^{2j-r-s}\ (q-1)^2 \delta_{i+4,r} \delta_{i+2,s-1} \delta_{s+1,r-1} \delta_{r+1,j}\ \alpha\ \delta_{j,i+2} \delta_{s,i+3} \delta_{r,i+4} \\
&&&\\
&& + & \sum_{s=i+2}^{j} \sum_{r=i+3}^{s}q^{2j-r-s}\ (q-1)^2 \delta_{i+4,r} \delta_{i+3,r-1} \delta_{r+1,s} \delta_{i+2,s-1} \delta_{s+1,j}\ \alpha\ \delta_{j,i+2} \delta_{s,i+3}
\end{array}
\]

\end{example}

Applying Lemma~\ref{conj} to the one gap word $\tau^{k_{0,i}}_{0,i}\cdot t_{j}^{k_j}$, where $k_j \in \mathbb{Z}\backslash \{0\}$ and $\alpha \in {\rm H}_n(q)$ we obtain:

$$\tau^{k_{0,i}}_{0,i}\cdot t_{j}^{k_j} \alpha \widehat{=}\left\{\begin{matrix}
\sum_{\lambda}{\tau^{k_{0,i}}_{0,i}t_{i+1}^{\lambda_{i+1}}\ldots t_{i+k_j}^{\lambda_{i+k_j}}\alpha^{\prime}} & {\rm if} \ k_j<s_{i,j} \\
&\\
\sum_{\lambda}{\tau^{k_{0,i}}_{0,i}t_{i+1}^{\lambda_{i+1}}\ldots t_j^{\lambda_j}}\beta^{\prime} &  {\rm if} \ k_j \geq s_{i,j}
\end{matrix}\right. ,
$$

where $\alpha^{\prime}, \beta^{\prime} \in \textrm{H}_{n}(q)$, $\sum_{\mu=i+1}^{i+k_j}{\lambda_{\mu}}= k_j$, $\lambda_{\mu} \geq 0,\ \forall \mu$ and if $\lambda_u=0$, then $\lambda_v=0$, $\forall v \geq u$.

\smallbreak

More precisely:

\begin{lemma}\label{gap}
For the 1-gap word $A\ =\ \tau_{0,i}^{k_{0,i}}\cdot t_j^{k_j}\cdot \alpha$, where $\alpha \in {\rm H}_n(q)$ we have:
\[
\begin{array}{lllcl}
(i) & {\rm If}\ |k_j|<s_{i,j},\ {\rm then}: & A & \widehat{=} & (q^{k_j-1})^{j-(i+1)}\tau_{0,i}^{k_{0,i}}\cdot t_{i+1}^{k_j}\ \delta_{i+2,j}\ \alpha\ \delta_{j,i+2}\ +\\
&&&&\\
&&&+& \sum_{\sum k_{i+1,i+k_j}=k_j} f(q,z) \tau_{0,i}^{k_{0,i}}\cdot \tau_{i+1,i+k_j}^{k_{i+1,i+k_j}} \cdot \beta\ \alpha\ \beta^{\prime}. \\
&&&&\\
(ii) & {\rm If}\ |k_j| \geq s_{i,j},\ {\rm then}: & A & \widehat{=} & (q^{k_j-1})^{j-(i+1)}\tau_{0,i}^{k_{0,i}}\cdot t_{i+1}^{k_j}\ \delta_{i+2,j}\ \alpha\ \delta_{j,i+2}\ +\\
&&&&\\
&&&+& \sum_{\sum k_{i+1,j}=k_j} f(q) \tau_{0,i}^{k_{0,i}}\cdot \tau_{i+1,j}^{k_{i+1,j}} \cdot \beta \alpha \beta^{\prime}.\\
\end{array}
\]
\noindent where $\beta$ and $\beta^{\prime}$ are of the form $w_{i+1,j} \in {\rm H}_{j+1}(q)$ and $\sum k_{i+1,k_j} = k_j$ such that $|k_{i+1}|< |k_j|$ and if $k_{\mu}=0$, then $k_{s}=0,$ for all $s>\mu$.
\end{lemma}

\begin{proof}
We prove the relations by induction on $k_j$. Let $0<k_j< j-i$.

\noindent For $k_j=1$ we have $A\ \widehat{=}\ \left[q^{(1-1)}\right]^{j-(i+1)}\ \tau_{0,i}^{k_{0,i}}\cdot t_{i+1} \delta_{i+2,j}\ \alpha\ \delta_{j,i+2}$ (Lemma~\ref{gapsimple}). Suppose that the relation holds for $k_j-1>1$. Then for $k_j$ we have:

$$
\begin{array}{lllcl}
A & = & \tau_{0,i}^{k_{0,i}}\cdot t^{k_j-1}_j\cdot (t_j\ \alpha) & \underset{ind.step}{\widehat{=}}& \underset{B}{\underbrace{\left[q^{k_j-2} \right]^{j-(i+1)} \tau_{0,i}^{k_{0,i}}\cdot t_{i+1}^{k_j-1} \underline{\delta_{i+2,j}\ t_j}\ \alpha \ \delta_{j,i+2}}}\ +\\
&&&&\\
&&& + & \underset{C}{\underbrace{ \sum_{k_{i_1,i+k_j-1}} f(q,z) \tau_{0,i}^{k_{0,i}} \cdot \tau_{i+1,i+k_j-1}^{k_{i_1,i+k_j-1}} \underline{\beta \ t_j}\ \beta^{\prime}}}.\\
\end{array}
$$

\noindent We now consider $B$ and $C$ separately and apply Lemma~\ref{brlre3} to both expressions:

\[
\begin{array}{cl}
B & \overset{(L.~\ref{brlre3})}{=} \\
&\\
= & \left[q^{k_j-2} \right]^{j-(i+1)} \tau_{0,i}^{k_{0,i}}\cdot t_{i+1}^{k_j-1}\left[ (q-1) \sum_{k+i+2}^{j} q^{j-k} t_k \delta_{i+2,k-1} \delta_{k+1,j} + q^{j-(i+2)+1} t_{i+1} \delta_{i+2,j}\right] \alpha \delta_{j,i+2}\\
&\\
= & \left[q^{k_j-2} \right]^{j-(i+1)}(q-1) \tau_{0,i}^{k_{0,i}}t_{i+1}\cdot \sum_{k+i+2}^{j}q^{j-k}t_k \delta_{i+2,k-1}\delta_{k+1,j}\alpha \delta_{j,i+2}\ +\\
&\\
+ & \left[q^{k_j-1} \right]^{j-(i+1)}  \tau_{0,i}^{k_{0,i}}\cdot t_{i+1}^{k_j} \delta_{i+2,j}\alpha \delta_{j,i+2}.\\
\end{array}
\]

\noindent We now do conjugation on the $\left(j-(i+3)\right)$-one gap words that occur and since
$t_k\cdot \beta\ \widehat{=}\ t_{i+2}\cdot \delta_{i+3,k}\ \beta\ \delta_{k,i+3}$ we obtain:

\[
\begin{array}{lcl}
B & \widehat{=} & \left[q^{k_j-1} \right]^{j-(i+1)} \tau_{0,i}^{k_{0,i}}\cdot t_{i+1}^{k_j} \delta_{i+2,j}\ \alpha\ \delta_{j,i+2}\ +\\
&&\\
& + & \tau_{0,i}^{k_{0,i}} t_{i+1}t_{i+2}\sum_{k=i+2}^{j} f(q,z) \delta_{i+3,k}\delta_{i+2,k-1}\delta_{k+1,j} \alpha \delta_{j,i+2}\delta_{k,i+3}\ =\\
&&\\
& = & \left[q^{k_j-1} \right]^{j-(i+1)} \tau_{0,i}^{k_{0,i}}\cdot t_{i+1}^{k_j} \delta_{i+2,j}\ \alpha\ \delta_{j,i+2}\ +\ \tau_{0,i}^{k_i} t_{i+1}t_{i+2}\cdot \beta_1,\\
\end{array}
\]

\noindent where $\beta_1 \in \textrm{H}_{j+1}(q)$.

\smallbreak

\noindent Moreover, $C\ = \ \sum_{k_r} f(q) \tau_{0,i}^{k_{0,i}} \cdot \tau_{i+1,i+k_j-1}^{k_{i+1,i+k_j-1}} \beta \ t_j\ \beta^{\prime}$ and since $\beta\ =\ w_{i+k_j-1,j}$, we have that: $\beta\cdot t_j\ \overset{(L.~\ref{brlre3})}{=}\ \sum_{s=i+k_j-1}^{j} t_s\cdot \gamma_s$, where $\gamma_s \in \textrm{H}_{j+1}(q)$ and so: $C\ \widehat{=}\ \sum_{v_r}f(q) \tau_{0,i}^{k_{0,i}}\cdot \tau_{i+1, i+k_j}^{v_{i+1,i+k_j}}\cdot \beta_2$, where $\beta_2 \in \textrm{H}_{j+1}(q)$.

\smallbreak

This concludes the proof.
\end{proof}

We now pass to the general case of one-gap words.

\begin{prop}\label{fp}
For the $1$-gap word $B\ =\ \tau_{0,i}^{k_{0,i}} \cdot \tau_{j,j+m}^{k_{j,j+m}} \cdot \alpha$, where $\alpha \in {\rm H}_n(q)$ we have:
\[
\begin{array}{lcl}
B & \widehat{=} & \prod_{s=0}^{m}{(q^{k_{j+s}-1})^{j-(i+1)}}\cdot \tau_{0,i}^{k_{0,i}} \tau_{i+1,i+m}^{k_{j,j+m}} \\
& & \\
& \cdot & \prod_{s=0}^{m}(\delta_{i+m+2-s,j+s}) \cdot \alpha \cdot \prod_{s=0}^{m}(\delta_{j+s,i+m+2-s})\ +\\
& & \\
&+&\sum_{u_r} f(q) \tau_{0,i}^{k_{0,i}}\cdot (\tau_{i+1,i+m}^{u_{1,m}})\cdot \alpha^{\prime} \\
\end{array}
\]
\noindent where $\alpha^{\prime}\in {\rm H}_n(q)$, $\sum u_{1,m} = k_j$ such that $u_1< k_j$ and if $u_{\mu}=0$, then $u_{s}=0, \forall s>\mu$.
\end{prop}

\begin{proof}
The proof follows from Lemma~\ref{gap}. The idea is to apply Lemma~\ref{gap} on the expression $\underline{\tau_{0,i}^{k_{0,i}}\cdot t_j^{k_j}} \cdot \rho_1$, where $\rho_1\ =\ \tau_{j+1,j+m}^{k_{j+1,j+m}}$ and obtain the terms $\tau_{0,i}^{k_{0,i}}\cdot t_{i+1}^{k_j}\cdot \rho_2$ and $\tau_{0,i}^{k_{0,i}}\cdot \tau_{i+1,i+q}^{k_{i+1,i+q}} \cdot \rho_2$ and follow the same procedure until there is no gap in the word.
\end{proof}

We are now ready to deal with the general case, that is words with more than one gap in the indices of the generators.

\begin{thm}\label{fth}
For the $\phi$-gap word $C\ =\ \tau^{k_{0,i}}_{0,i}\cdot \tau^{k_{i+s_1,i+s_1+\mu_1}}_{i+s_1,i+s_1+\mu_1}\cdot \tau^{k_{i+s_2,i+s_2+\mu_2}}_{i+s_2,i+s_2+\mu_2} \ldots \tau^{k_{i+s_{\phi},i+s_{\phi}+\mu_{\phi}}}_{i+s_{\phi},i+s_{\phi}+\mu_{\phi}} \cdot \alpha$, where $k_i \in \mathbb{Z}\backslash \{0\}$ for all $i$, $\alpha \in {\rm H}_n(q)$, $s_j, \mu_j \in \mathbb{N}$, such that $s_1>1$ and $s_j>s_{j-1}+\mu_{j-1}$ for all $j$ we have:

\[
\begin{array}{lcl}
C & \widehat{=} & \prod_{j=1}^{\phi}{\left(q^{k_{i+s_j}-1} \right)^{s_j-j- \sum_{p=1}^{j-1}\mu_p}}\cdot \tau^{u_{0,i+\phi+ \sum_{p=1}^{\phi}{\mu_p}}}_{0,i+\phi+ \sum_{p=1}^{\phi}{\mu_p}}\cdot \left( \prod_{p=0}^{\phi -1}\alpha_{\phi-p} \right) \cdot \alpha \cdot \left( \prod_{p=1}^{\phi}\alpha^{\prime}_{p} \right)\ +\\
&&\\
& + &  \sum_{v} f_v(q) \tau^{k_{0,v}}_{0,v}\cdot w_v, \  \mbox{\rm where}  \\
\end{array}
\]

\begin{itemize}
\item[(i)] $\alpha_j\ =\ \prod_{\lambda_j=0}^{\mu_j}{\delta_{i+j+1+\sum_{k=1}^{j}{\mu_k}-\lambda_j,\ i+s_j+ \mu_j- \lambda_j}},\ j\ =\ \{1,2,\ldots , \phi\},$
\vspace{.1in}
\item[(ii)] $\alpha^{\prime}_j\ =\ \prod_{\lambda_j=0}^{\mu_j}{\delta_{i+j+1+\sum_{k=1}^{j-1}{\mu_k}+\lambda_j,\ i+s_j+ \lambda_j}},\ j\ =\ \{1,2,\ldots , \phi\},$
\vspace{.1in}
\item[(iii)] $\tau^{u_{0,i+\phi+ \sum_{p=1}^{\phi}{\mu_p}}}_{0,i+\phi+ \sum_{p=1}^{\phi}{\mu_p}}\ = \ \tau_{0,i}^{k_{0,i}}\cdot \prod_{j=1}^{\phi}{\tau^{k_{i+s_j,i+s_j+\mu_j}}_{i+j+\sum_{p=1}^{j-1}{\mu_p},i+j+\sum_{p=1}^{j}{\mu_p}}},$
\vspace{.1in}
\item[(iv)] $\tau^{u_{0,v}}_{0,v}< \tau^{u_{0,i+\phi+\sum_{p=1}^{\phi}\mu_p}}_{0,i+\phi+\sum_{p=1}^{\phi}\mu_p},$ for all $v$,
\vspace{.1in}
\item[(v)] $w_v$ of the form  $w_{i+2,i+s_{\phi}+\mu_{\phi}} \in {\rm H}_{i+s_{\phi}+\mu_{\phi}+1}(q)$, for all $v$,
\vspace{.1in}
\item[(vi)] the scalars $f_v(q)$ are expressions of $q\in \mathbb{C}$ for all $v$.
\end{itemize}

\end{thm}

\begin{proof}
We prove the relations by induction on the number of gaps. For the $1$-gap word $\tau^{k_{0,i}}_{0,i}\cdot \tau^{k_{i+s,i+s+\mu}}_{i+s,i+s+\mu}\cdot \alpha$, where $\alpha \in \textrm{H}_n(q)$, we have:

$$
\begin{array}{lcl}
A & \widehat{=} & \left[ \prod_{\lambda=0}^{\mu}{\left(q^{k_{i+s+\lambda}-1} \right)^{s-1}}\right] \cdot \tau^{k_{0,i}}_{0,i}\cdot \tau^{k_{i+s,i+s+\mu}}_{i+1,i+1+\mu}\cdot \prod_{\lambda=0}^{\mu}{\delta_{i+2+\mu-\lambda, i+s+\mu - \lambda}}\cdot \alpha \cdot \prod_{\lambda=0}^{\mu}{\delta_{i+2+\mu+\lambda, i+s+ \lambda}}\ +\\
&&\\
& + & \sum_{v}{f_v(q)\cdot \tau^{u_{0,v}}_{0,v}\cdot w_v}, \mbox{ \rm which holds from Proposition~\ref{fp}.}  \\
\end{array}
$$

\smallbreak

Suppose that the relation holds for $(\phi -1)$-gap words. Then for a $\phi$-gap word we have:

\smallbreak

\noindent $\left(\underline{\tau^{k_{0,i}}_{0,i}\cdot \tau^{k_{i+s_1,i+s_1+\mu_1}}_{i+s_1,i+s_1+\mu_1}\cdot \tau^{k_{i+s_2,i+s_2+\mu_2}}_{i+s_2,i+s_2+\mu_2} \ldots \tau^{k_{i+s_{\phi -1},i+s_{\phi -1}+\mu_{\phi -1}}}_{i+s_{\phi -1},i+s_{\phi -1}+\mu_{\phi -1}}}\right) \cdot \tau^{k_{i+s_{\phi},i+s_{\phi}+\mu_{\phi}}}_{i+s_{\phi},i+s_{\phi}+\mu_{\phi}} \cdot \alpha \underset{ind.step}{\widehat{=}}$\\

\noindent $\prod_{j=1}^{\phi-1}\left(q^{k_{i+s_j}-1} \right)^{s_j-j-\sum_{k=1}^{j-1}\mu_k}\cdot \tau_{0,i+\phi-1+\sum_{k=1}^{\phi-1}\mu_k}^{u_{0,i+\phi-1+\sum_{k=1}^{\phi-1}\mu_k}} \cdot \underline{\prod_{k=0}^{\phi-2}\alpha_{\phi-1-k}\cdot \tau_{i+s_{\phi},i+s_{\phi}+\mu_{\phi}}^{k_{i+s_{\phi},i+s_{\phi}+\mu_{\phi}}}} \cdot \alpha \cdot \prod_{k=1}^{\phi-1}\alpha^{\prime}_k\ +$\\

\noindent $\sum_{v}f_v(q)\cdot \tau_{0,v}^{u_{0,v}}\cdot \underline{w \cdot \tau_{i+s_{\phi},i+s_{\phi}+\mu_{\phi}}^{k_{i+s_{\phi},i+s_{\phi}+\mu_{\phi}}}}\ \overset{s_{\phi}>s_{\phi-1} + \mu_{\phi-1}}{=}$\\

\noindent $\prod_{j=1}^{\phi-1}\left(q^{k_{i+s_j}-1} \right)^{s_j-j-\sum_{k=1}^{j-1}\mu_k}\cdot \underline{\tau_{0,i+\phi-1+\sum_{k=1}^{\phi-1}\mu_k}^{u_{0,i+\phi-1+\sum_{k=1}^{\phi-1}\mu_k}} \cdot \tau_{i+s_{\phi},i+s_{\phi}+\mu_{\phi}}^{k_{i+s_{\phi},i+s_{\phi}+\mu_{\phi}}}} \cdot \prod_{k=0}^{\phi-2}\alpha_{\phi-1-k} \cdot \alpha \cdot \prod_{k=1}^{\phi-1}\alpha^{\prime}_k\ +$\\

\noindent $\sum_{v}f_v(q)\cdot {\tau_{0,v}^{u_{0,v}}\cdot \tau_{i+s_{\phi},i+s_{\phi}+\mu_{\phi}}^{k_{i+s_{\phi},i+s_{\phi}+\mu_{\phi}}}}\cdot w \ \overset{(Prop.~\ref{fp})}{=}$\\

\noindent $\prod_{j=1}^{\phi-1}\left(q^{k_{i+s_j}-1} \right)^{s_j-j-\sum_{k=1}^{j-1}\mu_k}\cdot \prod_{p=0}^{\mu_{\phi}}\left(q^{k_{i+s_{\phi}+p}-1} \right)^{s_{\phi}-\phi-\sum_{k=1}^{\phi-1}\mu_k} \tau_{0,i+\phi-1+\sum_{k=1}^{\phi-1}\mu_k}^{u_{0,i+\phi-1+\sum_{k=1}^{\phi-1}\mu_k}} \cdot $\\

\noindent $\tau_{i+{\phi}+\sum_{k=1}^{\phi-1}{\mu_k},i+{\phi}+\sum_{k=1}^{\phi-1}{\mu_k}+\mu_{\phi}}^{k_{i+s_{\phi},i+s_{\phi}+\mu_{\phi}}} \cdot \prod_{k=0}^{\phi-1}\alpha_{\phi-1-k} \cdot \alpha \cdot \prod_{k=1}^{\phi-1}\alpha^{\prime}_k + \sum_{v}f_v(q)\cdot \underline{\tau_{0,v}^{u_{0,v}}\cdot
\tau_{i+s_{\phi},i+s_{\phi}+\mu_{\phi}}^{k_{i+s_{\phi},i+s_{\phi}+\mu_{\phi}}}}\cdot w \ \overset{(Prop.~\ref{fp})}{=}$\\

\noindent $ \left[ \prod_{\lambda=0}^{\mu}{\left(q^{k_{i+s+\lambda}-1} \right)^{s-1}}\right] \cdot \tau^{k_{0,i}}_{0,i}\cdot \tau^{k_{i+s,i+s+\mu}}_{i+1,i+1+\mu}\cdot \prod_{\lambda=0}^{\mu}{\delta_{i+2+\mu-\lambda, i+s+\mu - \lambda}}\cdot \alpha \cdot \prod_{\lambda=0}^{\mu}{\delta_{i+2+\mu+\lambda, i+s+ \lambda}}\ +\\$

\noindent $ \sum_{v}{f_v(q)\cdot \tau^{u_{0,v}}_{0,v}\cdot w_v}.$

\end{proof}

All results are best demonstrated in the following example on a word with two gaps.

\begin{example}\rm For the 2-gap word $t^{k_0}t_1^{k_1}t_3t_{5}^2t_6^{-1}$ we have:\\

\noindent $ t^{k_0}\underline{t_1^{k_1}t_3}t_{5}^2t_6^{-1}\ =\ t^{k_0}t_1^{k_1}\underline{g_3t_2g_3}t_{5}^2t_6^{-1}\ =\ g_3t^{k_0}t_1^{k_1}t_2t_{5}^2t_6^{-1}g_3 \ \widehat{=} \ t^{k_0}t_1^{k_1}\underline{t_2t_{5}^2}t_6^{-1}g_3^2\ = $\\

\noindent $=\ t^{k_0}t_1^{k_1}t_2\underline{t_{5}}t_{5}t_6^{-1}g_3^2 \ = \ t^{k_0}t_1^{k_1}t_2\underline{g_5g_4}t_3g_4g_5t_{5}t_6^{-1}g_3^2\ =\ \underline{g_5g_4}t^{k_0}t_1^{k_1}t_2t_3g_4g_5t_{5}t_6^{-1}g_3^2\ \widehat{=}\\$

\noindent $\widehat{=}\ t^{k_0}t_1^{k_1}t_2t_3\underline{g_4g_5t_{5}}t_6^{-1}g_3^2g_5g_4\ =\ t^{k_0}t_1^{k_1}t_2t_3 \left[q^2t_3g_4g_5+q(q-1)t_4g_5 \ + \ (q-1)t_5g_4  \right]t_6^{-1}g_3^2g_5g_4\ =$ \\

\noindent $=\ q^2t^{k_0}t_1^{k_1}t_2t_3^2\underline{g_4g_5t_6^{-1}}g_3^2g_5g_4\ +\ q(q-1)t^{k_0}t_1^{k_1}t_2t_3t_4\underline{g_5t_6^{-1}}g_3^2g_5g_4\ +\ (q-1)t^{k_0}t_1^{k_1}t_2t_3t_5\underline{g_4t_6^{-1}}g_3^2g_5g_4\ =$\\

\noindent $=\ q^2t^{k_0}t_1^{k_1}t_2t_3^2\underline{t_6^{-1}}g_4g_5g_3^2g_5g_4\ + \ (q-1)t^{k_0}t_1^{k_1}t_2t_3\underline{t_5}t_6^{-1}g_4g_3^2g_5g_4 \ + \ q(q-1)t^{k_0}t_1^{k_1}t_2t_3t_4\underline{t_6^{-1}}g_5g_3^2g_5g_4\ \widehat{=}$\\

\noindent $\widehat{=}\ q^2t^{k_0}t_1^{k_1}t_2t_3^2\underline{g_6^{-1}g_5^{-1}}t_4^{-1}g_5^{-1}g_6^{-1}g_4g_5g_3^2g_5g_4\ + \ q(q-1)t^{k_0}t_1^{k_1}t_2t_3t_4\underline{g_6^{-1}}t_5^{-1}g_6^{-1} g_5g_3^2g_5g_4$\\

\noindent $+\ (q-1)t^{k_0}t_1^{k_1}t_2t_3\underline{g_5}{t_4}\underline{g_5}t_6^{-1}g_4g_3^2g_5g_4 \ = \ q^2g_6^{-1}g_5^{-1}t^{k_0}t_1^{k_1}t_2t_3^2t_4^{-1}g_5^{-1}g_6^{-1}g_4g_5g_3^2g_5g_4\ +$\\

\noindent $+\ q(q-1)g_6^{-1}t^{k_0}t_1^{k_1}t_2t_3t_4t_5^{-1}g_6^{-1} g_5g_3^2g_5g_4 \ +\ (q-1){g_5}t^{k_0}t_1^{k_1}t_2t_3{t_4}t_6^{-1}{g_5}g_4g_3^2g_5g_4 \ \widehat{=}$\\

\noindent $\widehat{=}\ q^2t^{k_0}t_1^{k_1}t_2t_3^2t_4^{-1}g_5^{-1}g_6^{-1}g_4g_5g_3^2g_5g_4g_6^{-1}g_5^{-1}\ +\ q(q-1)t^{k_0}t_1^{k_1}t_2t_3t_4t_5^{-1}g_6^{-1} g_5g_3^2g_5g_4g_6^{-1}\ +$\\

\noindent $+\ (q-1)t^{k_0}t_1^{k_1}t_2t_3{t_4}\underline{t_6^{-1}}{g_5}g_4g_3^2g_5g_4{g_5} \ =\ q^2t^{k_0}t_1^{k_1}t_2t_3^2t_4^{-1}g_5^{-1}g_6^{-1}g_4g_5g_3^2g_5g_4g_6^{-1}g_5^{-1}\ +$\\

\noindent $+\ q(q-1)t^{k_0}t_1^{k_1}t_2t_3t_4t_5^{-1}g_6^{-1} g_5g_3^2g_5g_4g_6^{-1} \ +\ (q-1)t^{k_0}t_1^{k_1}t_2t_3{t_4}\underline{g_6^{-1}}t_5^{-1}g_6^{-1}{g_5}g_4g_3^2g_5g_4{g_5} \ \widehat{=}$\\

\noindent $\widehat{=}\ q^2t^{k_0}t_1^{k_1}t_2t_3^2t_4^{-1}g_5^{-1}g_6^{-1}g_4g_5g_3^2g_5g_4g_6^{-1}g_5^{-1}\ + \ q(q-1)t^{k_0}t_1^{k_1}t_2t_3t_4t_5^{-1}g_6^{-1} g_5g_3^2g_5g_4g_6^{-1} \ +$\\

\noindent $+\ (q-1)t^{k_0}t_1^{k_1}t_2t_3{t_4}t_5^{-1}g_6^{-1}{g_5}g_4g_3^2g_5g_4{g_5}{g_6^{-1}}.$

\end{example}

\subsection{Eliminating the tails}

So far we have seen how to convert elements of the basis $\Lambda^{\prime}$ to linear combinations of elements of $\Sigma_n$ and then, using conjugation, how these elements are expressed to linear combinations of elements of the $\bigcup_n\textrm{H}_n(q)$-module $\Lambda$. We will show now that using conjugation and stabilization moves all these elements of the $\bigcup_n\textrm{H}_n(q)$-module $\Lambda$ are expressed to linear combinations of elements in the set $\Lambda$ with scalars in the field $\mathbb{C}$. We will use the symbol $\simeq$ when a stabilization move is performed and $\widehat{\simeq}$ when both stabilization moves and conjugation are performed.

\smallbreak

Let us consider a generic word in $\textrm{H}_{1,n+1}(q)$. This is of the form $\tau_{0,n}^{k_{0,n}}\cdot w_{n+1}$, where $w_{n+1} \in \textrm{H}_{n+1}(q)$. Without loss of generality we consider the exponent of the braiding generator with the highest index to be $(-1)$ when the exponent of the corresponding loop generator is in $\mathbb{N}$ and $(+1)$ when the exponent of the corresponding loop generator is in $\mathbb{Z} \backslash \mathbb{N}$. We then apply Lemma~\ref{brlre2} and \ref{brlre3} in order to interact $t_n^{\pm k_n}$ with $g_n^{\mp 1}$ and obtain words of the following form:

\[
\begin{array}{llll}
(1) & \tau_{0,p}^{\lambda_{0,p}}\cdot v, & \textrm{where} & \tau_{0,p}^{\lambda_{0,p}} < \tau_{0,n}^{k_{0,n}}\ \textrm{and}\ v \in \textrm{H}_{n+1}(q)\ \textrm{of\ any\ length,\ or}\\
&&&\\
(2) & \tau_{0,q}^{k_{0,q}}\cdot u, & \textrm{where} & \tau_{0,q}^{\lambda_{0,q}} < \tau_{0,n}^{k_{0,n}}\ \textrm{and}\ u \in \textrm{H}_{n}(q)\ \textrm{such that}\ l(u)<l(w).\\
\end{array}
\]

In the first case we obtain monomials of $t_i$s of less order than the initial monomial, followed by a word in ${\rm H}_{n+1}(q)$ of any length. After at most $(k_n+1)$-interactions of $t_n$ with $g_n$, the exponent of $t_n$ will become zero and so by applying a stabilization move we obtain monomials of $t_i$s of less index, and thus of less order (Definition~\ref{order}), followed by a word in ${\rm H}_{n}(q)$.

In the second case, we have monomials of $t_i$s of less order than the initial monomial followed by words $u \in \textrm{H}_{n}(q)$ such that $l(u) < l(w)$. We interact the generator with the maximum index of $u$, $g_m$ with the corresponding loop generator until the exponent of $t_m$ becomes zero. A gap in the indices of the monomials of the $t_i$s occurs and we apply Theorem~\ref{fth}. This leads to monomials of $t_i$s of less order followed by words of the braiding generators of any length. We then apply stabilization moves and repeat the same procedure until the braiding `tails' are eliminated.

\begin{thm} \label{tails}
Applying conjugation and stabilization moves on a word in the $\bigcup_{\infty}{\rm H}_n(q)$-module, $\Lambda$ we have that:

$$\tau_{0,m}^{k_{0,m}}\cdot w_n\ \widehat{\simeq}\ f(q,z)\cdot \sum_{j}{f_j(q,z)\cdot \tau_{0,u_j}^{v_{0,u_j}}},$$

\noindent such that $\sum{v_{0,u_j}}=\sum{k_{0,m}}$ and $\tau_{0,u_j}^{v_{0,u_j}} < \tau_{0,m}^{k_{0,m}}$, for all $j$.
\end{thm}

The logic for the induction hypothesis is explained above. We shall now proceed with the proof of the theorem.

\begin{proof}
We prove the statement by double induction on the length of $w_n\in \textrm{H}_n(q)$ and on the order of $\tau_{0,m}^{k_{0,m}}\in \Lambda$, where order of $\tau_{0,m}^{k_{0,m}}$ denotes the position of $\tau_{0,m}^{k_{0,m}}$ in $\Lambda$ with respect to total-ordering.

\smallbreak

For $l(w)=0$, that is for $w=e$ we have that $\tau_{0,m}^{k_{0,m}}\ \widehat{\simeq}\ \tau_{0,m}^{k_{0,m}}$ and there's nothing to show. Moreover, the minimal element in the set $\Lambda$ is $t^k$ and for any word $w\in \textrm{H}_n(q)$ we have that $t^k \cdot w\ \simeq\ f(q,z)\cdot t^k$, by the quadratic relation and stabilization moves.

\smallbreak

Suppose that the relation holds for all $\tau_{0,p}^{u_{0,p}}\cdot w^{\prime}$, where $\tau_{0,p}^{u_{0,p}} \leq \tau_{0,m}^{k_{0,m}}$ and $l(w^{\prime})=l$, and for all $\tau_{0,q}^{v_{0,q}}\cdot w$, where $\tau_{0,q}^{v_{0,q}} < \tau_{0,m}^{k_{0,m}}$ and $l(w)=l+1$. We will show that it holds for $\tau_{0,m}^{k_{0,m}}\cdot w$. Let the exponent of $t_r$, $k_r \in \mathbb{N}$ and let $w\in \textrm{H}_{r+1}(q)$. Then, $w$ can be written as $w^{\prime} \cdot g_r^{-1} \cdot \delta_{r-1,d}$, where $w^{\prime} \in \textrm{H}_{r}(q)$ and $d<r$. We have that:

\[
\begin{array}{lcl}
\tau_{0,m}^{k_{0,m}}\cdot w & = & \tau_{0,r-1}^{k_{0,r-1}} t_r^{k_r-1} \tau_{r+1,m}^{k_{r+1,m}} \cdot w^{\prime} \cdot t_r g_r^{-1} \delta_{r-1,d}\ =\\
&&\\
& = & \tau_{0,r-1}^{k_{0,r-1}} t_r^{k_r-1} \tau_{r+1,m}^{k_{r+1,m}} \cdot w^{\prime} \cdot g_r \underline{t_{r-1} \delta_{r-1,d}}\ \overset{L.~6}{=}\\
&&\\
& = & \tau_{0,r-1}^{k_{0,r-1}} t_r^{k_r-1} \tau_{r+1,m}^{k_{r+1,m}} \cdot w^{\prime} \cdot g_r\\
&&\\
&& \cdot  \left( \sum_{j=0}^{r-1-d}{q^j(q-1)\delta_{r-1,\widehat{r-1-j},d}t_{r-1-j}}\ +\ q^{l(\delta_{r-1,d})}\delta_{r-1,d}t_{d-1} \right) \widehat{=}\\
&&\\
& \widehat{=} & \sum_{j=0}^{r-1-d}{q^j(q-1) \tau_{0,r-1}^{k_{0,r-1}} t_r^{k_r-1} \tau_{r+1,m}^{k_{r+1,m}}\cdot t_{r-1-j}}\cdot w^{\prime}\cdot g_r \delta_{r-1,\widehat{r-1-j},d}\ +\\
&&\\
& + & q^{l(\delta_{r-1,d})} \tau_{0,r-1}^{k_{0,r-1}} t_r^{k_r-1} \tau_{r+1,m}^{k_{r+1,m}}\cdot t_{d-1} \cdot w.\\
\end{array}
\]

\bigbreak

We have that $\left(\tau_{0,r-1}^{k_{0,r-1}} t_r^{k_r-1} \tau_{r+1,m}^{k_{r+1,m}}\cdot t_{r-1-j}\right) < \left(t_{0,m}^{k_{0,m}} \right)$, for all $j\in \{1,2,\ldots r-1-d \}$ and $l\left(w^{\prime}\cdot g_r \delta_{r-1,\widehat{r-1-j},d}\right)=l$ and $ \left( \tau_{0,r-1}^{k_{0,r-1}} t_r^{k_r-1} \tau_{r+1,m}^{k_{r+1,m}}\cdot t_{d-1}\right) < \left(t_{0,m}^{k_{0,m}} \right)$. So, by the induction hypothesis, the relation holds.

\end{proof}

\begin{example}\rm In this example we demonstrate how to eliminate the braiding `tail' in a word in $\Sigma_n$.

\[
\begin{array}{lcl}
t^{-1}\underline{t_1^2}t_{2}^{-1}g_1^{-1} & = & t^{-1}t_1t_{2}^{-1}\underline{t_1g_1^{-1}}\ = \ t^{-1}t_1t_{2}^{-1} g_1 \underline{t} \ \widehat{=} \ \underline{t_1}t_{2}^{-1} g_1\ =\ t_{2}^{-1}\underline{t_1 g_1}\ =\\
&&\\
& = &  (q-1)  \underline{t_1}t_{2}^{-1}\ +\ q t_{2}^{-1} g_1 \underline{t}\ \widehat{=}\ (q-1) t\underline{t_2^{-1}}g_1^{2}\ +\ qt \underline{t_{2}^{-1}} g_1\ =\\
&&\\
& = & (q-1) tt_1^{-1}g_2^{-1}g_1^{2}g_2^{-1}\ +\ q tt_1^{-1}g_2^{-1}g_1g_2^{-1}.\\
\end{array}
\]

\noindent We have that:

\[
\begin{array}{lcl}
g_2^{-1}g_1g_2^{-1} & = & q^{-2}g_1g_2g_1\ +\ q^{-1}(q^{-1}-1)g_2g_1\ +\  q^{-1}(q^{-1}-1) g_1g_2\ +\ (q^{-1}-1)^2g_1,\\
&&\\
g_2^{-1}g_1^2g_2^{-1} & = & q^{-2}(q-1)g_1g_2g_1\ -\ (q^{-1}-1)^2g_2g_1\ -\ (q^{-1}-1)^2 g_1g_2\ +\ (q-1)(q^{-1}-1)^2g_1\\
&&\\
& + & q(q^{-1}-1)g_2^{-1}\ +\ 1,\\
\end{array}
\]

\noindent and so

\[
\begin{array}{rcl}
(q-1) tt_1^{-1}g_2^{-1}g_1^{2}g_2^{-1} & \widehat{\simeq} & \left( (q-1)+ q^{-1}(q-1)^3 \right)\cdot tt_1^{-1}\ -\ q^{-3}(q^{-1}-1)^3z^2\cdot 1\ +\\
&&\\
& + & 3q^{-3}(q-1)^4z\cdot 1\ - \ q^{-1}(q-1)^2z\cdot 1\ -\ q^{-3}(q-1)^5\cdot 1,\\
&&\\
q tt_1^{-1}g_2^{-1}g_1g_2^{-1} & \widehat{\simeq} & z \cdot tt_1^{-1}\ +\ q^{-1}(q^{-1}-1)z^2\cdot 1\ +\ 2(q^{-1}-1)^2z\cdot 1\ +\ q(q^{-1}-1)^3\cdot 1.\\
\end{array}
\]

\end{example}

\section{The basis $\Lambda$ of $\mathcal{S}({\rm ST})$}

In this section we shall show that the set $\Lambda$ is a basis for $\mathcal{S}({\rm ST})$, given that $\Lambda^{\prime}$ is a basis of $\mathcal{S}({\rm ST})$. This is done in two steps:

\smallbreak

$\bullet$ We first relate the two sets $\Lambda$ and $\Lambda^{\prime}$ via an infinite lower triangular matrix with invertible elements in the diagonal. Since $\Lambda^{\prime}$ is a basis for $\mathcal{S}({\rm ST})$, the set $\Lambda$ spans $\mathcal{S}({\rm ST})$.

\smallbreak

$\bullet$ Then, we prove that the set $\Lambda$ is linear independent and so we conclude that $\Lambda$ forms a basis for $\mathcal{S}({\rm ST})$.

\subsection{The infinite matrix}

With the orderings given in Definition~\ref{order} we shall show that the infinite matrix converting elements of the basis $\Lambda^{\prime}$ to elements of the set $\Lambda$ is a block diagonal matrix, where each block is an infinite lower triangular matrix with invertible elements in the diagonal. Note that applying conjugation and stabilization moves on an element of some $\Lambda_k$ followed by a braiding part won't alter the sum of the exponents of the loop generators and thus, the resulted terms will belong to the set of the same level $\Lambda_k$. Fixing the level $k$ of a subset of $\Lambda^{\prime}$, the proof of Theorem~\ref{mainthm} is equivalent to proving the following claims:

\smallbreak

\begin{itemize}
\item[(1)] A monomial $w^{\prime} \in \Lambda_k^{\prime} \subseteq \Lambda^{\prime}$ can be expressed as linear combinations of elements of $\Lambda_k \subseteq \Lambda$, $v_i$, followed by monomials in $\textrm{H}_n(q)$, with scalars in $\mathbb{C}$ such that $\exists \ j: v_j=w\sim w^{\prime}$.
\smallbreak
\item[(2)] Applying conjugation and stabilization moves on all $v_i$'s results in obtaining elements in $\Lambda_k$, $u_i$'s, such that $u_i < v_i$ for all $i$.
\smallbreak
\item[(3)] The coefficient of $w$ is an invertible element in $\mathbb{C}$.
\smallbreak
\item[(4)] $ \Lambda_{k} \ni w < u \in \Lambda_{k+1}$.
\end{itemize}

\bigbreak

Indeed we have the following: Let $w^{\prime} \in S^{\prime}_k \subseteq \Lambda^{\prime}$. Then, by Theorem~7 the monomial $w^{\prime}$ is expressed to linear combinations of elements of $\Sigma_n$, where the only term that isn't followed by a braiding part is the homologous monomial $w \in \Lambda$. Other terms in the linear combinations involve lower order terms than $w$ (with possible gaps in the indices) followed by a braiding part and words of the form $w \cdot \beta$, where $\beta \in \textrm{H}_n(q)$. Then, by Theorem~8 elements of $\Sigma_n$ are expressed to linear combinations of elements of the $\textrm{H}_{n}(q)$-module $\Lambda$ (regularizing elements with gaps) and obtaining words which are of less order than the initial word $w$. In Theorem~9 all elements who are followed by a braiding part are expressed as linear combinations of elements of $\Lambda$ with coefficients in $\mathbb{C}$. It is essential to mention that when applying Theorem~9 to a word of the form $w\cdot \beta$ one obtains elements in $\Lambda$ that are less ordered that $w$. Thus, we obtain a lower triangular matrix with entries in the diagonal of the form $q^{-A}$ (see Theorem~7), which are invertible elements in $\mathbb{C}$. The fourth claim follows directly from Definition~\ref{order}.

\smallbreak

If we denote as $[\Lambda_k]$ the block matrix converting elements in $\Lambda^{\prime}_k$ to elements in $\Lambda_k$ for some $k$, then the change of basis matrix will be of the form:

$$
S=\left[
\begin{array}{ccccccc}
\ddots & 0 & 0 & 0 & 0 & 0 &  \\
 & [\Lambda_{k-2}] & 0 & 0 & 0 & 0 &  \\
& 0 & [\Lambda_{k-1}] & 0 & 0 & 0 &  \\
& 0 & 0 & [\Lambda_{k}] & 0 & 0 &  \\
& 0 & 0 & 0 & [\Lambda_{k+1}] & 0 &  \\
& 0 & 0 & 0 & 0 & [\Lambda_{k+2}] &  \\
& 0 & 0 & 0 & 0 & 0 & \ddots \\
\end{array}\right]
$$

\smallbreak

\begin{center}
The infinite block diagonal matrix
\end{center}

\subsection{Linear independence of $\Lambda$}

Consider an arbitrary subset of $\Lambda$ with finite many elements $\tau_1, \tau_2, \ldots, \tau_k$. Without loss of generality we consider $\tau_1 < \tau_2 < \ldots < \tau_k$ according to Definition~\ref{order}. We convert now each element $\tau_i \in \Lambda$ to linear combination of elements in $\Lambda^{\prime}$ according to the infinite matrix. We have that

$$\tau_i\ \widehat{\simeq}\ A_i \tau_i^{\prime}\ +\ \sum_{j}A_j \tau_j^{\prime}\ ,$$

\noindent where $\tau_i^{\prime} \sim \tau_i$, $A_i \in \mathbb{C}\setminus \{0\}$, $\tau_j^{\prime} < t_i^{\prime}$ and $A_j \in \mathbb{C}, \forall j$.

\smallbreak

So, we have that:

\[
\begin{array}{ccl}
\tau_1 & \widehat{\simeq} & A_1 \tau_1^{\prime} + \sum_{j}A_{1j} \tau_{1j}^{\prime}\\
&&\\
\tau_2 & \widehat{\simeq} & A_2 \tau_2^{\prime} + \sum_{j}A_{2j} \tau_{2j}^{\prime}\\
&&\\
\vdots && \ \ \ \ \ \ \ \ \vdots \\
&&\\
\tau_{k-1} & \widehat{\simeq} & A_{k-1} \tau_{k-1}^{\prime} + \sum_{j}A_{(k-1)j} \tau_{(k-1)j}^{\prime}\\
&&\\
\tau_k & \widehat{\simeq} & A_k \tau_k^{\prime} + \sum_{j}A_{kj} \tau_{kj}^{\prime}\\
\end{array}
\]

\smallbreak

Note that each $\tau_i^{\prime}$ can occur as an element in the sum $\sum_{j}A_{pj} \tau_{pj}^{\prime}$ for $p > i$. We consider now the equation  $\sum_{i=1}^{k}\lambda_i \cdot \tau_i\ =\ 0\ ,\ \lambda_i\in \mathbb{C}, \forall i$ and we show that this holds only when $\lambda_i=0, \forall i$. Indeed, we have:

$$\sum_{i=1}^{k}\lambda_i \cdot \tau_i\ =\ 0\ \Leftrightarrow\ \lambda_k A_k \tau_k^{\prime}\ +\ \sum_{i=1}^{k}\sum_{j}\lambda_i A_ij \tau_{ij}^{\prime}\ =\ 0,$$

\noindent where $\tau_k^{\prime} > \tau_{ij}^{\prime}, \forall i, j$. So we conclude that $\lambda_k \ =\ 0$. Using the same argument we have that:

$$\sum_{i=1}^{k}\lambda_i \cdot \tau_i\ =\ 0\ \Leftrightarrow\ \sum_{i=1}^{k-1}\lambda_i \cdot \tau_i\ =\ 0\ \Leftrightarrow\ \lambda_{k-1} A_{k-1} \tau_{k-1}^{\prime}\ +\ \sum_{i=1}^{k-1}\sum_{j}\lambda_i A_ij \tau_{ij}^{\prime}\ =\ 0,$$

\noindent where $\tau_{k-1}^{\prime} > \tau_{ij}^{\prime}, \forall i, j$. So, $\lambda_{k-1} \ =\ 0$. Retrospectively we get:

$$ \sum_{i=1}^{k}\lambda_i \cdot \tau_i\ =\ 0\  \Leftrightarrow\  \lambda_i\ =\ 0,\ \forall i,$$

\noindent and so an arbitrary finite subset of $\Lambda$ is linear independent. Thus, the set $\Lambda$ is linear independent and it forms a basis for $\mathcal{S} (\rm ST)$.

\smallbreak

The proof of our main theorem is now concluded. \hfill QED

\section{Conclusions}

In this paper we gave a new basis $\Lambda$ for $\mathcal{S} (\rm ST)$, different from the Turaev-Hoste-Kidwell basis and the Morton-Aiston basis. This basis has been conjectured by J.H.~Przytycki. The new basis is appropriate for describing the handle sliding moves, whilst the old basis $\Lambda^{\prime}$ is consistent with the trace rules \cite{La2}. In a sequel paper we shall use the bases $\Lambda^{\prime}$ and $\Lambda$ of $\mathcal{S}({\rm ST})$ and the change~of~basis~matrix in order to compute the Homflypt skein module of the lens spaces $L(p,1)$.




\begin{thebibliography}{99}

\bibitem[B]{B} {\sc C. Blanchet}, Hecke algebras, modular categories and $3$-manifolds quantum invariants, {\it Topology} {\bf 39} (2000), No. 1, 193-223.

\bibitem[D]{D} {\sc I. Diamantis}, The Homflypt skein module of lens spaces, PhD thesis, National Technical University of Athens, in preparation.

\bibitem[DL]{DL} {\sc I. Diamantis, S. Lambropoulou}, Braid equivalence in $3$-manifolds with rational surgery description, arXiv:1311.2465 [math.GT], (2013).

\bibitem[HK]{HK} {\sc J.~Hoste, M.~Kidwell}, Dichromatic link invariants, {\it Trans. Amer. Math. Soc.} {\bf 321} (1990), No. 1, 197-229.

\bibitem[HP]{HP} {\sc J.~Hoste, J.~Przytycki}, An invariant of dichromatic links, {\it Proc. of the AMS} {\bf 105} (1989), No. 4, 1003-1007.

\bibitem[Jo]{Jo} {\sc V.~F.~R.~Jones}, Hecke algebra representations of braid groups and link polynomials, {\it Ann. Math.} {\bf 126} (1987), 335-388.

\bibitem[KL]{KL} {\sc D. Kodokostas, S. Lambropoulou}, Knot theory for $3$-manifolds represented by the $2$-unlink, in preperation.

\bibitem[La1]{La1} {\sc S. Lambropoulou}, Solid torus links and Hecke algebras of B-type, {\it Quantum Topology}; D.N. Yetter Ed.; World Scientific Press, (1994), 225-245.

\bibitem[La2]{La2} {\sc S. Lambropoulou}, Knot theory related to generalized and cyclotomic Hecke algebras of type B, {\it J. Knot Theory Ramifications} {\bf 8} (1999), No. 5, 621-658.

\bibitem[LG]{LG} {\sc S. Lambropoulou, M. Geck}, Markov traces and knot invariants related to the Iwahori-Hecke algebras of type B, {\it J. $f\ddot{u}r$ die reine und angewandte Mathematic} {\bf 482} (1997), 191-213.

\bibitem[LR1]{LR1} {\sc S. Lambropoulou, C.P. Rourke}, Markov's theorem in $3$-manifolds, \emph{Topology and its Applications} {\bf 78}, (1997), 95-122.

\bibitem[LR2]{LR2} {\sc S. Lambropoulou, C. P. Rourke}, Algebraic Markov equivalence for links in $3$-manifolds, {\emph Compositio Math.} {\bf 142} (2006), 1039-1062.

\bibitem[MA]{MA} {\sc H. Morton, A. Aiston}, Young diagrams, the Homfly skein of the annulus and unitary invariants, Proc. of Knots 96, edited by S.~Suzuki, World Scientific (1997), 31-45.

\bibitem[P]{P} {\sc J.~Przytycki}, Skein modules of 3-manifolds, {\it Bull. Pol. Acad. Sci.: Math.}, {\bf 39, 1-2} (1991), 91-100.

\bibitem[P2]{P2} {\sc J.~Przytycki}, Skein module of links in a handlebody, {\it Topology 90, Proc. of the Research Semester in Low Dimensional Topology at OSU}, Editors: B.Apanasov, W.D.Neumann, A.W.Reid, L.Siebenmann, De Gruyter Verlag (1992), 315-342.


\bibitem[Tu]{Tu} {\sc V.G.~Turaev}, The Conway and Kauffman modules of the solid torus,  {\it Zap. Nauchn. Sem. Lomi} {\bf 167} (1988), 79--89. English translation: {\it J. Soviet Math.} (1990), 2799-2805.

\end{thebibliography}
\end{document}